\pdfoutput=1
\documentclass[a4paper,final,fleqn]{article}

\usepackage[lm,grey]{janm}
\usepackage{math}
\usepackage{marginnote}

\DeclareMathAlphabet{\mathpzc}{OT1}{pzc}{m}{it}

\usepackage{showkeys}
\usepackage[textwidth=3.5cm,color=lightgray]{todonotes}

\newcommand{\ld}{\text{\tiny\ensuremath{\bullet}}}
\newcommand{\dDl}{\Dl^{\ld}}

\renewcommand{\H}{H}
\DeclareMathOperator{\arccosh}{arccosh}
\newcommand{\asin}{\sin^{-1}}
\newcommand{\Wp}{\Ws}

\hypersetup{
  pdfinfo={
    Title={On Two-Sided Estimates for the Nonlinear Fourier Transform of KdV},
    Author={Jan-Cornelius Molnar},
    Subject={Two-Sided Estimates for the Nonlinear Fourier Transform of KdV},
    Keywords={nonlinear Fourier transform, Korteweg-de Vries equation, integrable PDEs}
  }
}

\title{On Two-Sided Estimates for the Nonlinear Fourier Transform of {K}d{V}}
\author{Jan-Cornelius Molnar\footnote{Supported in part by the Swiss National Science Foundation}}
\date{\today}

\begin{document}

\maketitle

\begin{abstract}
The KdV-equation $u_t = -u_{xxx} + 6uu_x$ on the circle admits a global nonlinear Fourier transform, also known as Birkhoff map, linearizing the KdV flow. The regularity properties of $u$ are known to be closely related to the decay properties of the corresponding nonlinear Fourier coefficients. In this paper we obtain two-sided polynomial estimates of all integer Sobolev norms $||u||_m$, $m\ge 0$, in terms of the weighted norms of the nonlinear Fourier transformed, which are linear in the highest order. We further obtain quantitative estimates of the nonlinear Fourier transformed in arbitrary weighted Sobolev spaces.



\end{abstract}

% ===============================================================================================
\section{Introduction}

We consider the \emph{\kdv equation}
\[
  u_t = -u_{xxx} + 6uu_x
\]
on the circle $\T = \R/\Z$ with $u$ real-valued. As is well known, the \kdv equation can be written as an infinite-dimensional Hamiltonian system
\[
  \partial_{t}u = \partial_{x} \partial_{u} \H,\qquad \H(u) = \frac{1}{2}\int_\T (u_x^2 + 2u^3)\,\dx,
\]
where we choose the standard Sobolev space $\Hs^{m} \defl \Hs^{m}(\T,\R)$, $m\ge 0$, for the phase space, with the $L^{2}$-inner product
\[
  \lin{u,v} \defl \int_{\T} uv \,\dx,
\]
and endow this space with the Poisson structure proposed by Gardner
\[
  \setd{F,G}
   =
  \int_\T \partial_{u}F\partial_{x}\partial_{u}G\,\dx.
\]
The mean value $[u]$ is a Casimir for this bracket and hence is preserved by the \kdv flow. On the invariant subspace
\[
  \Hs^{m}_{0} \defl \setdef{u\in \Hs^{m}}{[u] = 0},
\]
the Poisson structure is nondegenerate and the \kdv evolution is given by
\[
  u_t = \setd{u,H}.
\]

As an infinite-dimensional Hamiltonian system, the \kdv equation  is well known to be completely integrable. According to \cite{Kappeler:2003up}, this is true in the strongest possible sense meaning that \kdv admits \emph{global Birkhoff coordinates} $(x_{n},y_{n})_{n\ge 1}$. To give a precise statement of our results, we define for $u = \sum_{n\in\Z} u_{n}\e^{\ii 2n\pi x}$ in $\Hs_{0}^{m}$ the Sobolev norm $\n{u}_{m}$ by
\[
  \n{u}_{m}^{2} \defl \sum_{n\in\Z} \lin{2n\pi}^{2m} \abs{u_{n}}^{2},\qquad
   \lin{\al} \defl 1+\abs{\al}.
\]
Further, we introduce the model space $h_{\star}^{m} \defl \ell_{m+1/2}^{2}(\N,\R)\times \ell_{m+1/2}^{2}(\N,\R)$ with elements $(x,y) = (x_{n},y_{n})_{n\ge 1}$ where the norm $\n{(x,y)}_{h_{\star}^{m}}$ is defined by
\[
  \n{(x,y)}_{h_{\star}^{m}}^{2} \defl \sum_{n\ge 1} (2n\pi)^{2m+1}(x_{n}^{2}+y_{n}^{2}).
\]
This space is endowed with the canonical Poisson structure $\{x_{n},y_{m}\} = -\{y_{m},x_{n}\} = \dl_{nm}$ while all other brackets vanish.

The \emph{Birkhoff map} $u\mapsto (x_{n},y_{n})_{n \ge 1}$ is a bi-analytic, canonical diffeomorphism $\Om\colon \Hs^{0}_{0}\to h^{0}_{\star}$, whose restriction $\Om\colon \Hs^{m}_{0}\to h_{\star}^{m}$ is again bi-analytic for any integer $m\ge 1$. On $h^{1}_{\star}$ the transformed \kdv Hamiltonian $H\circ\Om^{-1}$ is a real-analytic function of the actions $I_{n} = (x_{n}^{2}+y_{n}^{2})/2$ alone and the equations of motion take the particular simple form
\[
  \dot x_{n} = -\om_{n} y_{n},\quad
  \dot y_{n} =  \om_{n} x_{n},\qquad
  \om_{n} \defl \partial_{I_{n}} \H.
\]
The mapping $\Om$ may thus be viewed as a nonlinear Fourier transform for the \kdv equation. Furthermore,  the derivative $\ddd_{0}\Om$ of $\Om$ at the origin \emph{is} a weighted Fourier transform, and on $\Hs_{0}^{0}$
\begin{equation}
  \label{p-id}
  \n{\Om(u)}_{h_{\star}^{0}} = \n{u}_{0},
\end{equation}
which we refer to as Parseval's identity -- cf. e.g. \cite{Kappeler:2003up}. Our main result says that for higher order Sobolev norms the following version of Parseval's identity holds for the nonlinear map $\Om$.

\begin{thm}
\label{b-est}
For any integer $m\ge 1$ there exist absolute constants $c_{m}$, $d_{m} > 0$ such that
the restriction of $\Om$ to $\Hs_{0}^{m}$ satisfies the two-sided estimates
\[
  \emph{ (i)}\quad
  \n{\Om(u)}_{h_{\star}^{m}} \le c_{m}\bigl(
  	\n{u}_{m} + (1+\n{u}_{m-1})^{m}\n{u}_{m-1}	\bigr),
\]
and
\[
  \emph{(ii)}\quad
  \n{u}_m \le d_{m}\bigl(
  	\n{\Om(u)}_{h_{\star}^{m}} +
	(1+\n{\Om(u)}_{h_{\star}^{m-1}})^{m}\n{\Om(u)}_{h_{\star}^{m-1}} \bigr).\fish
\]
\end{thm}

The estimates (i) and (ii) are reminiscent of the 1-smoothing property of the Birkhoff map $\Om$ established in \cite{Kappeler:2012fa} as they are linear in the highest Sobolev norm $\n{u}_{m}$ and the highest weighted $h_{\star}^{m}$-norm $\n{\Om(u)}_{m}$, respectively.

\begin{cor}
\label{s-est}
Suppose $u(t)$ is a solution of the \kdv equation with initial value $u_{0}\in \Hs^{m}_{0}$, $m\in \Z_{\ge 1}$. Then there exists an absolute constant $\al_{m}$ such that for all time $t\in\R$,
\[
  \n{u(t)}_{m} \le \al_{m}\left(\n{u_{0}}_{m} + (1+\n{u_{0}}_{m-1})^{m^{2}+m-1}\n{u_{0}}_{m-1}\right).\fish
\]
\end{cor}

The proof of Theorem~\ref{b-est} relies on estimates for the \kdv action variables $I(u) = (I_n)_{n\ge 1}$ where $I_{n} = (x_{n}^{2} + y_{n}^{2})/2$, $n\ge 1$. The decay properties of the actions are  closely related to the regularity properties of $u$ -- cf. \cite{Kappeler:1999er,Kappeler:2001hsa,Djakov:2006ba,Poschel:2011iua}. We quantify this relationship by providing two-sided estimates of the Sobolev norms of $u$ in terms of weighted $\ell^{1}$-norms of $I(u)$. For that purpose introduce for $s$ real the weighted sequence space $\ell^{1}_{s}$ whose norm is defined by
\[
  \n{I(u)}_{\ell^{1}_{s}} \defl \sum_{n \ge 1} (2n\pi)^{s} I_{n}.
\]

\begin{thm}
\label{act-sob-est}
For any integer $m\ge 1$ there exist absolute constants $c_m$, $d_{m} > 0$, such that
\[
  \emph{ (i)}\quad
  \n{I(u)}_{\ell^{1}_{2m+1}}
   \le c_{m}^{2}\Bigl(\n{u}_{m}^{2} + (1+\n{u}_{m-1})^{2m}\n{u}_{m-1}^{2}\Bigr),
\]
as well as
\[
  \emph{(ii)}\quad
  \n{u}_{m}^{2} \le d_{m}^{2}\Bigl(\n{I(u)}_{\ell^{1}_{2m+1}}
                   + (1+\n{I(u)}_{\ell^{1}_{2m-1}})^m\n{I(u)}_{\ell^{1}_{2m-1}}\Bigr),
\]
for all $u\in\Hs_{0}^{m}$.\fish
\end{thm}

Note that in the estimates (i) and (ii) of Theorem~\ref{act-sob-est} the corresponding exponents are of the same order since the action variables have to be viewed as quadratic quantities.

For the direct problem, that is estimates (i) of Theorems~\ref{b-est} \& \ref{act-sob-est}, we obtain estimates which hold for a larger family of spaces referred to as \emph{weighted Sobolev spaces} -- see \cite{Kappeler:1999er,Kappeler:2001hsa} for an introduction. A \emph{normalized, symmetric, submultiplicative}, and \emph{monotone weight} is a function $w\colon\Z\to\R$  with
\[
  w_n \ge 1,\qquad w_{n} = w_{-n},\qquad w_{n+m}\le w_{n}w_{m},\qquad
  w_{\abs{n}}\le w_{\abs{n}+1},
\]
for all $n,m\in\Z$. The class of all such weights is denoted by $\Ms$ and $\Hs_{0}^w$ is the space of $\Hs_{0}^{0}$ functions $u$ with finite $w$-norm
\[
  \n{u}_w^2 \defl \sum_{n\in\Z} w_{2n}^2 \abs{u_{n}}^{2}.
\]
Further, $h_{\star}^{w}\subset h_{\star}^{0}$ denotes the elements $(x,y)$ with $\n{x}_{w}^{2} + \n{y}_{w}^{2} < \infty$, where
\[
  \n{x}_{w}^{2} \defl \sum_{n\ge 1} (2n\pi)w_{2n}^{2}\abs{x_{n}}^{2}.
\]

For any $s \ge 0$, the \emph{Sobolev weight} $\lin{n\pi}^{s}$ gives rise to the usual Sobolev space $\Hs^{s}_{0}$. For $s\ge 0$ and $a > 0$, the \emph{Abel weight} $\lin{n\pi}^s \e^{a\abs{n}}$ gives rise to the space $\Hs_{0}^{s,a}$ of $L^2$-functions, which can be analytically extended to the open strip $\setdef{z}{\abs{\Im z} < a/2\pi}$ of the complex plane with traces in $\Hs_{0}^s$ on the boundary lines. In between are, among others, the \emph{Gevrey weights}
\[
  \lin{n}^s \e^{a\abs{n}^\sigma},\qquad 0< \sigma < 1,\quad s\ge 0,\quad a > 0,
\]
which give rise to the Gevrey spaces $\Hs^{s,a,\sigma}_{0}$, as well as weights of the form
\[
  \lin{n}^s\exp\biggl(\frac{a\abs{n}}{1+\log^{\sg}\lin{n}} \biggr),
    \qquad 0< \sigma < 1,\quad s\ge 0,\quad a > 0,
\]
that are lighter than Abel weights but heavier than Gevrey weights.

We assume all weights $w\in\Ms$ to be piecewise linearly extended to functions on the real line $w\colon \R\to [1,\infty)$, $t\mapsto w(t)$.

\begin{theorem}
\label{act-west}
For any weight $w = \lin{n}^{r}v$ with $0\le r\le 1/2$ and $v\in\Ms$ there exists an absolute constant $c_{w,r} > 0$ such that the restriction of the Birkhoff map $\Om$ to $\Hs_{0}^{w}$ takes values in $h_{\star}^{w}$ and satisfies
\[
  \emph{ (i)}\quad
  \n{\Om(u)}_{h_{\star}^{w}} \le c_{w,r}\left((1+\n{u}_{w}) + w(16\n{u}_{w}^{3/2-r})\right)\n{u}_{w}.
\]
Moreover, we find for the action variables on $\Hs_{0}^{w}$
\[
  \emph{(ii)}\quad\sum_{n\ge 1} w_{2n}^{2}(2n\pi)I_{n}
   \le
  c_{w,r}^{2}\left((1+\n{u}_{w})^{2} + \bigl(w(16\n{u}_{w}^{3/2-r})\bigr)^{2}\right)\n{u}_{w}^{2}.\fish
\]
\end{theorem}

The bounds given in (i) and (ii) are valid for all submultiplicative weights including those growing exponentially fast. They reflect the nature of the weight: If, for example, $w$ grows polynomially, then the bounds (i) and (ii) are polynomial in $\n{u}_{w}$, whereas if $w$ grows exponentially, then the bounds are exponential in $\n{u}_{w}$. Note that the bounds improve as soon as the weight $w$ incorporates the factor $\lin{n}^{r}$ with $r > 0$. The following version of Theorem~\ref{act-west} for Sobolev spaces of real exponent complements the results of Theorems~\ref{b-est}-\ref{act-sob-est}.

\begin{corollary}
For any real $s\ge 0$ there exist an absolute constant $c_{s}$ such that the restriction of the Birkhoff map $\Om$ to $\Hs^{s}_{0}$ takes values in $h_{\star}^{s}$ and satisfies
\[
  \emph{ (i)}\quad\n{\Om(u)}_{h_{\star}^{s}} \le c_{s}(1+\n{u}_{s})^{\max(1,s)}\n{u}_{s}.
\]
Moreover, we find for the action variables on $\Hs_{0}^{s}$
\[
  \emph{(ii)}\quad\n{I(u)}_{\ell^{1}_{2s+1}} \le c_{s}^{2}(1+\n{u}_{s})^{2\max(1,s)}\n{u}_{s}^{2}.\fish
\]
\end{corollary}

\emph{Method of proof.}
The main ingredient into the proof of Theorem~\ref{b-est} is the estimate of the action variables $I_{n}$, $n\ge 1$, of Theorem~\ref{act-sob-est}. To establish the latter, we use the \kdv action variables $J_{n,m}$, $n\ge 1$, on level $m\ge 0$ introduced by McKean \& Vaninsky~\cite{McKean:1997ka} in context of the nonlinear Schrödinger (NLS) equation. They are defined in terms of spectral data of the corresponding Hill operator
\begin{align}
  \label{sop}
  L(u) = -\frac{\ddd^2}{\ddd x^2} + u,
\end{align}
arising in the Lax-pair formulation of \kdv\ -- see Section~\ref{s:setup}.
For $m=0$, the actions $J_{n,0}$ coincide with $I_{n}$, whereas for $m\ge 1$ they have the asymptotic behavior
\[
  J_{n,m} \sim (2n\pi)^{2m} I_{n},\qquad n\to \infty.
\]
Furthermore, they satisfy the trace formula
\[
  \sum_{n\ge 1} (2n\pi) J_{n,m} = \frac{1}{4^m}H_m + O(H_{m-1},\dotsc,H_{0}),
  \qquad m\ge 0,
\]
where $\H_{k}$ denotes the $k$th Hamiltonian in the \emph{\kdv hierarchy}.
The first two Hamiltonians of this hierarchy are given by
\[
  H_{0}(u) = \frac{1}{2}\int_{\T} u^{2}\,\dx,
  \qquad
  H_{1}(u) = H(u) = \frac{1}{2}\int_{\T} ((\partial_{x}u)^{2} + 2 u^{3})\,\dx.
\]
More generally, these Hamiltonians have the form
\[
  H_m(u) = \frac{1}{2}\int_\T \p*{ \partial_{x}^{m}u)^2 + p_m(u,\ldots,\partial_{x}^{m-1}u) }\,\dx
\]
with $p_m$ being a canonically-determined polynomial.

Taking the asymptotic behavior of the actions $J_{n,m}$, the corresponding trace formulae, and the representation of the Hamiltonians together we get
\[
  \sum_{n\ge 1} (2n\pi)^{2m+1} I_{n} \sim
  \sum_{n\ge 1} (2n\pi) J_{n,m} \sim
  \frac{1}{2}\n{u}_{m}^{2} + O(\n{u}_{m-1}),
\]
which yields the bound stated in Theorem~\ref{act-sob-est} (i).
To obtain item (ii) of Theorem~\ref{act-sob-est} we rewrite the $m$th Hamiltonian as
\[
  \frac{1}{2}\n{\partial_{x}^{m}u}_{0}^{2}
   = H_{m} - \frac{1}{2}\int_{\T} p_m(u,\ldots,\partial_{x}^{m-1}u)\,\dx.
\]
Using the trace formula we derive a bound of $H_{m}$ in terms of $\n{I(u)}_{\ell_{2m+1}^{1}}$. Since the $p_{m}$-term depends only on derivatives of $u$ up to order $m-1$, an inductive argument then gives part (ii) of Theorem~\ref{act-sob-est}.
We remark that the asymptotics of $J_{n,m}$ are derived from sufficiently accurate asymptotics of the periodic eigenvalues of the corresponding Hill operator, and that our method of proof does not involve any auxiliary spectral quantities such as the spectral heights.

To prove Theorem~\ref{act-west}, we take a slightly different approach by estimating the action variables in terms of the spacing of the periodic eigenvalues of the associated Hill operator. For the latter, estimates in any weighted norm are available -- see e.g. \cite{Kappeler:1999er,Kappeler:2001hsa,Djakov:2006ba,Poschel:2011iua} -- allowing us to obtain Theorem~\ref{act-west}.

\emph{Related results.}
First results on global bounds of the \kdv actions in terms of Sobolev norms were obtained by Korotyaev~\cite{Korotyaev:2000tc,Korotyaev:2006uh} using conformal mapping theory.
More precisely, he shows \cite[Theorem 2.4 \& 2.6]{Korotyaev:2006uh} for any $m\ge 0$ that
$\n{I(u)}_{\ell^{1}_{2m+1}}$ is bounded by
\begin{align*}
  c_{m}^{2}
  \n{\partial_{x}^{m}u}_{0}^{2}
  (1+\n{\partial_{x}^{m}u}_{0})^{\frac{4m+2}{3}}
  \paren[\Big]{1+\n{\partial_{x}^{m}u}_{0}^{\frac{2m}{m+2}}(1+\n{\partial_{x}^{m}u}_{0})^{\frac{m(4m+2)}{3(m+2)}}}.
\end{align*}
By augmenting $c_{m}$ but without increasing the degree of the right hand side, this bound may be simplified to
\[
  \n{I(u)}_{\ell^{1}_{2m+1}}
   \le
  c_{m}^{2}
  (1+\n{u}_{m})^{\frac{8m+2}{3}}\n{u}_{m}^{2}.
\]
In comparison, our estimate Theorem~\ref{act-sob-est} (i) is linear in $\n{u}_{m}^{2}$ with the remainder  $(1+\n{u}_{m-1})^{2m}\n{u}_{m-1}^{2}$ involving only Sobolev norms of order $m-1$.

For the inverse problem Korotyaev obtains \cite[Theorem 2.4 \& 2.6]{Korotyaev:2006uh}
\[
  \n{\partial_{x}^{m}u}_{0}^{2}
   \le
  c_{m}
  (1+\n{I(u)}_{\ell^{1}_{2m+1}})^{\frac{4m+2}{3}+\frac{8m+10}{3}N_{m}}\n{I(u)}_{\ell^{1}_{2m+1}},
\]
where
\[
  N_{m} = \frac{m+1}{3}\left(1 + tm + t^{2}m(m-1) + \dotsb + t^{m+1}m!\right)\Big|_{t=2/3}.
\]
Note that $N_{m}$ grows factorially with $m$. Our estimate Theorem~\ref{act-sob-est} (ii) considerably improves the latter one, since the bound is linear in $\n{I(u)}_{\ell_{2m+1}^1}$, the exponent of the remainder $(1+\n{I(u)}_{\ell_{2m-1}^1})^{m}\n{I(u)}_{\ell_{2m-1}^{1}}$ is linear in $m$, and the remainder only involves $\ell^{1}$-norms of order $2m-1$.

A priori estimates of \kdv solutions obtained directly from the Hamiltonians in the \kdv hierarchy are well known -- see e.g. \cite{Bona:1975hn}. The application of Theorem~\ref{b-est} to the initial value problem of \kdv stated in Corollary~\ref{s-est} improves on these results since the bound of $\n{u(t)}_{m}$ is linear in $\n{u_{0}}_{m}$ and the exponent of the remainder $(1+\n{u_{0}}_{m-1})^{m^{2}+m-1}\n{u_{0}}_{m-1}$ grows only quadratically in $m$.

First results on bounds of the Birkhoff map for \kdv on the weighted Sobolev spaces $\Hs_{0}^{w}$ can be found in Kappeler \& P{\"o}schel~\cite{Kappeler:2009uk}. The authors proved
\[
  u\in \Hs_{0}^{w}\quad \iff \quad \Om(u) \in h_{\star}^{w},
\]
where the $\Leftarrow$ part holds true only for weights with subexponential growth. The estimates of $\n{\Om(u)}_{h_{\star}^{w}}$ presented in Theorem~\ref{act-west} quantify this relationship.

The methods developed in this paper were introduced by the author \cite{Molnar:2014vg} in the context of the nonlinear Fourier transform for the periodic defocusing \nls equation
\[
  \ii\partial_{t}u = -\partial_{x}^{2}u + 2\abs{u}^{2}u,\qquad x\in \T,\quad u\in \C,
\]
to show results corresponding to Theorems~\ref{b-est}-\ref{act-west}. We note that the result of Theorem~\ref{b-est} improves on the corresponding result obtained for the defocusing NLS equation since the exponents in the remainders of estimate (i) and estimate (ii) are equal.

From a broader perspective we may view the weighted actions $(2n\pi)I_{n}$ as a perturbation of the squared modulus $\abs{u_{n}}^{2}$ of the $n$th Fourier coefficient of $u$. Our method of comparing the norms $\n{I(u)}_{\ell^{1}_{2m+1}}$ with the Hamiltonians of the \kdv hierarchy consists in a separate analysis of Fourier modes of low and high frequencies. This idea has a long history in the analysis of nonlinear PDEs. Most recently, it lead Colliander, Keel, Staffilani, Takaoka \& Tao \cite{Colliander:2001wg,Colliander:2003fv,Colliander:2004gc} to invent the I-Method, which allows to obtain global well-posedness of subcritical equations in low-regularity regimes where the Hamiltonian (or other integrals) of the equation cease to be well defined. The idea is to damp all sufficiently high Fourier modes of a local solution such that the Hamiltonian can be controlled by weaker norms while still being an >>almost conserved<< quantity. The difficulty here is to choose the damping carefully enough such that the nonlinearity of the equation does not create a significant interaction of low and and high frequencies. Our aim is so to say opposite to that of the I-Method: As we look for quantitative global estimates, the most delicate part of our analysis is to get control of the modes of low frequencies. It is achieved by an appropriate localization of the periodic eigenvalues of the Hill operator -- see Proposition~\ref{ev-as} and the references mentioned in Appendix~A.

The results on the analyticity and asymptotic expansion of the integral $F(\lm)$ introduced in Section~\ref{s:setup} have also been applied in \cite{Kappeler:CNzeErmy} concerning convexity properties of the \kdv Hamiltonian. They were put together in a joint effort.

\emph{Organization of the paper.} In Section~\ref{s:setup} the \kdv action variables on integer levels $m\ge 0$ are defined and the trace formulae relating them to the hierarchy of \kdv Hamiltonians are proven. In Sections~\ref{s:act-est} and \ref{s:sob-est} Theorem~\ref{act-sob-est} (i) and (ii), respectively, are obtained by use of the localization of the Hill spectrum, which for the convenience of the reader is proved in Appendix~\ref{s:appendix}. Finally, in Section~\ref{s:act-west} we prove Theorem~\ref{act-west} by obtaining a uniform estimate of the actions in terms of the spacing of the periodic eigenvalues.

\emph{Acknowledgement.} The author is very grateful to Professor Thomas Kappeler for frequent discussions and valuable comments that greatly improved the manuscript.

% ===============================================================================================
\section{Setup}
\label{s:setup}

In this section we briefly recall the definition of the action variables as well as the main properties of the spectral quantities used to define them. We follow the exposition \cite{Kappeler:2003up} -- see also \cite{Poschel:1987uc,Poschel:2011iua,Djakov:2006ba,Flaschka:1976tc}. Consider \emph{Hill's operator}
\[
  L(q) = -\ddx + q,
\]
on the interval $[0,2]$, endowed with periodic boundary conditions and $q$ being a \emph{complex potential} in $\Hs^{0}_{0,\C} \defl \Hs^{0}_{0}(\T,\C)$. The spectrum of $L(q)$, called the \emph{periodic spectrum of $q$}, is pure point and complex in general as the operator is not self-adjoint. We may order the eigenvalues lexicographically -- first by their real part and second by their imaginary part -- such that
\[
  \lm_{0}^{+}(q) \lex \lm_{1}^{-}(q) \lex \lm_{1}^{+}(q) \lex \dotsb
          \lex \lm_{n}^{-}(q) \lex \lm_{n}^{+}(q) \lex \dotsb.
\]
Their asymptotic behavior is
\[
  \lm_{n}^{\pm}(q) = n^{2}\pi^{2} + \ell^{2}_{n},
\]
and we define for any $n\ge 1$ the \emph{gap length}
\[
  \gm_{n}(q) = \lm_{n}^{+}(q)-\lm_{n}^{-}(q) = \ell_{n}^{2}.
\]
Here $\ell_{n}^{2}$ stands for an $\ell^{2}$-sequence. For convenience we set $\gm_{0} = \infty$.

To obtain a suitable characterization of the periodic spectrum of $q$, we denote by $y_1(x,\lm,q)$ and $y_2(x,\lm,q)$ the standard fundamental solutions of $L(q)y = \lm y$, and by $\Delta(\lm,q)$ the \emph{discriminant}
\[
  \Delta(\lm,q) \defl y_1(1,\lm,q) + y_2'(1,\lm,q).
\]
To simplify matters, we may drop some or all of its arguments from the notation whenever there is no danger of confusion. The periodic spectrum of $q$ is precisely the zero set of the entire function $\Dl^2(\lm) - 4$, and we have the product representation
\[
  \Delta^2(\lm) - 4
   =
  -4(\lm-\lm_{0}^{+})\prod_{k\ge 1} \frac{(\lm_k^+-\lm)(\lm_k^--\lm)}{k^4\pi^4}.
\]
Hence, the discriminant is uniquely determined by the periodic spectrum. We also need the $\lm$-derivative $\dDl\defl\partial_{\lm}\Dl$ whose zeros are denoted by $\lm_{n}^{\ld}$ and satisfy $\lm_{n}^{\ld} = n^{2}\pi^{2} + \ell^{2}_{n}$. This derivative has the product representation
\[
  \dDl(\lm)  = -\prod_{m\ge 1} \frac{\lm_{m}^{\ld}-\lm}{m^{2}\pi^{2}}.
\]

For each real potential $q$ there exists an open neighborhood $\Wp_{q}$ within $\Hs^{0}_{0,\C}$ such that for every $p\in \Wp_{q}$ the closed intervals
\[
  G_{0} = \setdef{t+\lm_{0}^{+}}{-\infty < t \le 0},\qquad
  G_{n} = [\lm_{n}^{-},\lm_{n}^{+}],\quad n\ge 1,
\]
are disjoint from each other. Even more, there exist mutually disjoint neighborhoods $U_{n}\subset \C$, $n\ge 0$, called \emph{isolating neighborhoods}, which satisfy:

\begin{enumerate}[label=(\alph{*})]
\item $G_{n}$ and $\lm_{n}^{\ld}$ are contained in the interior of $U_{n}$ for every $p\in \Wp_{q}$,

\item there exists a constant $c \ge 1$ such that for $m\neq n$,
\[
  c^{-1}\abs{m^{2}-n^{2}} \le \dist(U_{n},U_{m}) \le c\abs{m^{2}-n^{2}},
\]

\item $U_{n} = \setd{\abs{\lm-n^{2}\pi^{2}} \le \pi/4}$ for $n$ sufficiently large.
\end{enumerate}

\noindent
Throughout this text $\Wp_{q}$ denotes a neighborhood of $q$ such that a common set of isolating neighborhoods for all $p\in \Wp_{q}$ exists. The union of all $\Wp_{q}$ for $q$ real defines an open and connected neighborhood of $\Hs^{0}_{0}$ within $\Hs^{0}_{0,\C}$ and is denoted by $\Wp$.

Following the approach of Flaschka \& McLaughlin~\cite{Flaschka:1976tc}, one can define action variables for the \kdv equation by Arnold's formula
\[
  I_n
   =
  \frac{1}{\pi}\int_{a_n}
  \frac{\lm \dDl(\lm)}{\sqrt{\Dl^2(\lm)-4}}
  \,\dlm,\qquad n\ge 1.
\]
Here $a_{n}$ denotes a cycle around $(\lm_{n}^{-},\lm_{n}^{+})$ on the spectral curve
\[
  C_{q} = \setdef{(\lm,z)}{z^2 = \Dl^2(\lm,q) - 4}\subset\C^{2},
\]
on which the square root $\sqrt{\Dl^2(\lm)-4}$ is defined. This curve is another spectral invariant associated with $q$, and an open Riemann surface of infinite genus if and only if the periodic spectrum of $q$ is simple. To avoid the technicalities involved with this curve, we fix proper branches of the square root which allows us to reduce the definition of the actions to standard contour integrals in the complex plane -- see also \cite{Kappeler:2003up,Kappeler:2005fb}.

Denote by $\sqrt[+]{\phantom{\lm}}$ the \emph{principal branch} of the square root on the complex plane minus the ray $(-\infty,0]$. Furthermore, for $q\in\Wp$ the \emph{standard root}
\[
  \vs_{n}(\lm) = \sqrt[\mathrm{s}]{(\lm_{n}^{+}-\lm)(\lm_{n}^{-}-\lm)},
                     \qquad \lm\notin G_{n},\qquad n\ge 1,
\]
is defined by the condition
\begin{align}
  \label{s-root}
  \vs_{n}(\lm) = (\tau_{n}-\lm)\sqrt[+]{1 - \gm_{n}^{2}/4(\tau_{n}-\lm)^{2}},
  						 \qquad \tau_{n} = (\lm_{n}^{-}+\lm_{n}^{+})/2.
\end{align}
The standard root is analytic in $\lm$ on $\C\setminus G_{n}$ and in $(\lm,p)$ on $(\C\setminus U_{n})\times \Wp_{q}$. Finally, we define the \emph{canonical root}
\[
  \sqrt[c]{\Dl^{2}(\lm)-4} \defl
   -2\ii\sqrt[+]{\lm-\lm_{0}^{+}}\prod_{k\ge 1} \frac{\vs_{k}(\lm)}{k^{2}\pi^{2}}.
\]
This root is analytic in $\lm$ on $\C\setminus\bigcup_{\atop{\gm_{k}\neq 0}{k\ge 0}} G_{k}$ and in $(\lm,p)$ on $(\C\setminus \bigcup_{k\ge 0} U_{k})\times \Wp_{q}$.

The \emph{$n$th \kdv action variable} of $q\in \Wp$ is then given by
\[
  \quad I_n
   \defl
  \frac{1}{\pi}\int_{\Gm_n}
  \frac{\lm \dDl(\lm)}{\sqrt[c]{\Dl^2(\lm)-4}}\,\dlm,
\]
where $\Gm_{n}$ denotes any sufficiently close circuit around $G_{n}$. More generally, we define the \emph{$n$th \kdv action variable on level $m\ge 0$} by
\[
  J_{n,m}
   \defl
  \frac{1}{(m+1)\pi}\int_{\Gm_n}
  \frac{\lm^{m+1}\dDl(\lm)}{\sqrt[c]{\Dl^{2}(\lm)-4}}\,\dlm.
\]
Note that the action on level zero $J_{n,0}$ equals the action $I_{n}$.

We proceed with the analysis of the analytical properties of the action integrand. To this end, we define for any $q\in \Wp$ on $(\C\setminus \bigcup_{k\ge0} U_{k}) \times \Wp_{q}$ the complex 1-form $\om(\lm) \equiv \om(\lm,p)$ by
\begin{align}
  \label{om}
  \om(\lm)
  \defl \frac{\dDl(\lm)}{\sqrt[c]{\Dl^{2}(\lm)-4}}\,\dlm
      = \frac{1}{2\ii}\frac{1}{\sqrt[+]{\lm-\lm_{0}^{+}}}
        \prod_{k\ge 1} \frac{\lm_{k}^{\ld}-\lm}{\vs_{k}(\lm)}\,\dlm.
\end{align}

A path in the complex plane is said to be \emph{admissible} for $q$ if, except possibly at its endpoints, it does not intersect any non collapsed gap $G_{n}(q)$.

\begin{lem}
\label{w-closed}
For each $q\in \Wp$, the 1-form $\om$ has the following properties:
\begin{equivenum}
\item
$\om$ is analytic on $(\C\setminus \bigcup_{k\ge 0} U_{k}) \times \Wp_{q}$,
\item
$\om(\lm,q)$ is analytic in $\lm$ on $\C\setminus \bigcup_{\atop{\gm_{k}\neq 0}{k\ge 0}} G_{k}$\footnote{With the convention $\gm_{0} = \infty$.}, and
\item
for $n\ge 1$ and any admissible path from $\lm_{n}^{-}$ to $\lm_{n}^{+}$ in $U_{n}$,
\[
  \int_{\lm_{n}^{-}}^{\lm_{n}^{+}} \om =  0.
\]
In particular, for any closed circuit $\Gm_{n}$ in $U_{n}$ around $G_{n}$,
\[
  \int_{\Gm_{n}} \om = 0.\fish
\]
\end{equivenum}
\end{lem}

\begin{proof}
(i)
$\dDl$ is analytic on $\C\times \Hs^{0}_{0,\C}$ and the canonical root is analytic on $(\C\setminus \bigcup_{k\ge 0} U_{k}) \times \Wp_{q}$ and does not vanish there. Therefore, $\om$ is analytic on $(\C\setminus \bigcup_{k\ge 0} U_{k}) \times \Wp_{q}$.

(ii)
Similarly, for $q$ fixed, $\om(\lm,q)$ is analytic in $\lm$ on $\C\setminus\bigcup_{k\ge 0} G_{k}$. Suppose $\gm_{n} = 0$ for some $n\ge 1$, then $\lm_{n}^{+} = \lm_{n}^{-}$ is a double root of $\Dl$, hence $\lm_{n}^{\ld} = \lm_{n}^{\pm}$ and the term $(\lm_{n}^{\ld}-\lm)/\vs_{n}(\lm)$ in the product representation \eqref{om} of $\om$ equals $1$, hence $\om(\lm,q)$ is analytic on $U_{n}$.

(iii)
We first consider the case of $q$ being real-valued. Clearly, the functional $\Wp_{q}\to \C$, $p\mapsto \int_{\partial U_{n}}\om$ is analytic. Further, for any $p\in \Wp_{q}$ real-valued, one has $(-1)^{n}\Dl(\lm) \ge 2$ on $G_{n}$ so, after deforming the contour of integration to $G_{n}$, according to the definition of the canonical root (see \cite{Kappeler:2003up} for a sign table)
\[
  \int_{\partial U_{n}}\om
   = 2\int_{\lm_{n}^{-}}^{\lm_{n}^{+}} \frac{(-1)^{n+1}\dDl(\lm)}{\sqrt[+]{\Dl^{2}(\lm)-4}}\,\dlm
   = -2\arccosh\frac{(-1)^{n}\Dl(\lm)}{2}\bigg|_{\lm_{n}^{-}}^{\lm_{n}^{+}} = 0.
\]
Thus $\int_{\partial U_{n}}\om$ vanishes on $\Wp_{q}\cap \Hs^{0}_{0}$ and hence on all of $\Wp_{q}$ by Lemma~\ref{ana-vanish}. In view of $\Wp = \bigcup_{q\in \Hs^{0}_{0}} \Wp_{q}$ we conclude that $\int_{\Gm_{n}}\om = 0$ for any closed circuit $\Gm_{n}$ in $U_{n}$ around $G_{n}$ and any $q\in \Wp$.

Now fix $q\in \Wp$ arbitrary. The identity $\int_{\lm_{n}^{-}}^{\lm_{n}^{+}} \om = 0$ clearly holds in the case $\lm_{n}^{+} = \lm_{n}^{-}$. If $\lm_{n}^{+} \neq \lm_{n}^{-}$, then we may define the left hand side $G_{n}^{+}$ and the right hand side $G_{n}^{-}$ of $G_{n}$ by
\[
  G_{n}^{\pm} \defl \setdef{\tau_{n} + (t\pm \ii 0)\gm_{n}/2\in \C}{-1\le t\le 1}.
\]
The canonical root admits opposite signs on $G_{n}^{\pm}$ that is
\[
  \sqrt[c]{\Dl^{2}(\lm)-4}\bigg|_{G_{n}^{-}} = -\sqrt[c]{\Dl^{2}(\lm)-4}\bigg|_{G_{n}^{+}}.
\]
Defining the contour $\Gm_{n}$ by going from $\lm_{n}^{-}$ to $\lm_{n}^{+}$ along $G_{n}^{-}$ and then going back from $\lm_{n}^{+}$ to $\lm_{n}^{-}$ along $G_{n}^{+}$ gives
\[
  0 = \int_{\Gm_{n}} \om
    = 2\int_{\lm_{n}^{-}}^{\lm_{n}^{+}}
       \frac{\dDl(\lm)}{\sqrt[c]{\Dl^{2}(\lm)-4}\big|_{G_{n}^{-}}}\,\dlm
    = -2\int_{\lm_{n}^{-}}^{\lm_{n}^{+}}
       \frac{\dDl(\lm)}{\sqrt[c]{\Dl^{2}(\lm)-4}\big|_{G_{n}^{+}}}\,\dlm.
\]
By contour deformation it then follows that $\int_{\lm_{n}^{-}}^{\lm_{n}^{+}} \om = 0$ along any admissible path in $U_{n}$.\qed
\end{proof}

Writing the action variables as
\begin{align}
  \label{Jn-om}
  J_{n,m} = \frac{1}{(m+1)\pi}\int_{\Gm_{n}} \lm^{m+1}\om,
\end{align}
makes the claimed analyticity on $\Wp$ evident. To proceed, we define for any $q\in \Wp$ and any $n\ge 0$ on $(\C\setminus \bigcup_{k\ge 0}U_{k})\times \Wp_{q}$ the primitive $F_{n}(\lm) \equiv F_{n}(\lm,p)$ by
\[
  F_{n}(\lm) \defl \int_{\lm_{n}^{+}}^{\lm} \om.
\]
Clearly, this improper integral exists, as $\om$ has the integrable singularity $1/\sqrt{\lm-\lm_{n}^{+}}$ near $\lm_{n}^{+}$. By Lemma~\ref{w-closed} the integral is also independent of the chosen admissible path.

\begin{lem}
\label{F-prop}
For every $q\in \Wp$ and every $n\ge 0$, we have that

\begin{equivenum}
\item $F_{n}$ is analytic in $(\lm,p)$ on $(\C\setminus\bigcup_{k\ge 0} U_{k}) \times \Wp_{q}$ and $F_{n}(\lm)\equiv F_{n}(\lm,q)$ is analytic in $\lm$ on $\C\setminus \bigcup_{\atop{\gm_{k}\neq 0}{k\ge 0}} G_{k}$.

\item $F_{0}(\lm) = F_{n}(\lm) - \ii n\pi$ on $\C\setminus \bigcup_{\atop{\gm_{k}\neq 0}{k\ge 0}} G_{k}$. In particular, $F_{0}$ extends continuously to all points $\lm_{0}^{+}$ and $\lm_{n}^{\pm}$, $n\ge 1$. One has $F_{0}(\lm_{0}^{+}) = 0$ and
\[
  F_{0}(\lm_{n}^{+}) = F_{0}(\lm_{n}^{-}) = - \ii n\pi,\qquad n\ge 1.
\]

\item $F_{n}^{2}(\lm)$ is analytic on $\C\setminus \bigcup_{\atop{0\le k \neq n}{\gm_{k}\neq 0}} G_{k}$ for every $n\ge 0$.

\item If $q$ is real, then for any $n\ge1$ and any real $\lm_{n}^{-} \le \lm \le \lm_{n}^{+}$,
\[
  F_{n}(\lm \pm \ii 0) = \pm f_{n}(\lm),
  \qquad f_{n}(\lm) = \arccosh \frac{(-1)^{n}\Dl(\lm)}{2}.
\]
Clearly, $f_{n}$ is continuous on $G_{n}$, strictly positive on $(\lm_{n}^{-},\lm_{n}^{+})$, and vanishes at the boundary points.
\item At the zero potential one has $F_{n}(\lm,0) = -\ii \sqrt[+]{\lm}+\ii n\pi$.\fish

\end{equivenum}

\end{lem}

\begin{proof}
(i)
Let us first consider the case where $n=0$. Since $\om$ is analytic on $(\C\setminus\bigcup_{k\ge 0} U_{k})\times \Wp_{q}$, it suffices to show that $F_{0}(\mu)$ is analytic on $\Wp_{q}$ for some fixed $\mu$ near $\partial U_{0}$. Choose the straight line $[\lm_{0}^{+},\mu]$ as the path of integration with parametrization
\[
  \lm_{t} = \lm_{0}^{+} + t(\mu-\lm_{0}^{+}).
\]
Then by the product representation \eqref{om} of $\om$
\[
  \int_{\lm_{0}^{+}}^{\mu} \om
   = \frac{1}{2\ii}\int_{0}^{1} \frac{\sqrt[+]{\mu-\lm_{0}^{+}}}{\sqrt[+]{t}} \chi_{0}(\lm_{t})\,\dt,\qquad
  \chi_{0}(\lm) = \prod_{k\ge 1} \frac{\lm_{k}^{\ld}-\lm}{\vs_{k}(\lm)}.
\]
Since $\lm_{0}^{+}$ is analytic on $\Wp$ and $\chi_{0}$ is analytic on $U_{0}\times \Wp_{q}$ -- cf. \cite[Appendix~A]{Kappeler:2005fb} and \cite[Corollary~12.8]{Grebert:2014iq} -- the claimed analyticity of $F_{0}$ follows.
In Lemma~\ref{F-ana} we obtain the analyticity of $F_{n}$, $n\ge 1$. This proof is a bit more work due to the fact that in this case $\lm_{n}^{+}$ may be a double eigenvalue.

(ii)
Note that $F_{0}(\lm) = F_{n}(\lm) + \int_{\lm_{0}^{+}}^{\lm_{n}^{+}}\om$.
Clearly, $F_{0}(\lm_{0}^{+}) = 0$ and $\int_{\lm_{n}^{-}}^{\lm_{n}^{+}}\om = 0$ for any $n\ge 1$ by Lemma~\ref{w-closed}, hence $F_{n}(\lm_{n}^{+}) = F_{n}(\lm_{n}^{-}) = 0$ and
\[
  \int_{\lm_{0}^{+}}^{\lm_{n}^{+}}\om = \sum_{k=0}^{n-1} \int_{\lm_{k}^{+}}^{\lm_{k+1}^{-}} \om.
\]
To proceed, first consider the case where $q$ is real-valued. 
In this case $\ii(-1)^{k}\sqrt[c]{\Dl^{2}(\lm)-4} > 0$ for $\lm_{k}^{+} < \lm < \lm_{k+1}^{-}$ -- see \cite[Section 5]{Kappeler:2003up} -- so
\begin{align*}
  \int_{\lm_{k}^+}^{\lm_{k+1}^-} \om &=
  \ii(-1)^{k}\int_{\lm_{k}^+}^{\lm_{k+1}^-}
  \frac{\dot\Dl(\lm)}{\sqrt[+]{4-\Dl^2(\lm)}}\,\dlm
   =
  \ii(-1)^{k}
  \asin\frac{\Dl(\lm)}{2}\bigg|_{\lm_{k}^+}^{\lm_{k+1}^-}
   =
  -\ii \pi,
\end{align*}
and hence $\int_{\lm_{0}^{+}}^{\lm_{n}^{+}}\om  = -\ii n\pi$.  The function $\int_{\lm_{0}^{+}}^{\lm_{n}^{+}}\om = F_{0}(\lm) - F_{n}(\lm)$ is analytic on $\Wp_{q}$ by item (i), therefore, $\int_{\lm_{0}^{+}}^{\lm_{n}^{+}}\om  = -\ii n\pi$ holds true on all of $\Wp_{q}$ in view of Lemma~\ref{ana-vanish}.

(iii)
In view of item (i) it remains to show that $F_{n}^{2}$ admits also for $\gm_{n}\neq 0$ an analytic extension from $U_{n}\setminus G_{n}$ to all of $U_{n}$. For $n=0$ write~\eqref{om} in the form
\[
  \om = \frac{1}{2\ii}\frac{1}{\sqrt[+]{\lm-\lm_{0}^{+}}}\chi_{0}(\lm)\,\dlm,
  \qquad
   \chi_{0}(\lm) =
        \prod_{k\ge 1} \frac{\lm_{k}^{\ld}-\lm}{\vs_{k}(\lm)},
\]
and similarly, for $n\ge 1$ write~\eqref{om} in the form
\[
  \om = \frac{1}{2\ii}\frac{\lm_{n}^{\ld}-\lm}{\vs_{n}(\lm)}\chi_{n}(\lm)\, \dlm,\qquad\;\;\; 
  \chi_{n}(\lm) = \frac{1}{\sqrt[+]{\lm-\lm_{0}^{+}}}
        \prod_{\atop{k\neq n}{k\ge 1}} \frac{\lm_{k}^{\ld}-\lm}{\vs_{k}(\lm)}.
\]
The mappings $\chi_{n}$, $n\ge 0$, are analytic on $U_{n}$ -- cf. \cite[Appendix A]{Kappeler:2005fb}. Moreover, the roots $\sqrt[+]{\lm-\lm_{0}^{+}}$ and $\vs_{n}(\lm)$, $n\ge 1$, respectively, admit opposite signs on opposite sides of $G_{0}$ and $G_{n}$, $n\ge 1$, respectively. Therefore, in view of $F_{n}(\lm) = \int_{\lm_{n}^{+}}^{\lm} \om$, for any $n\ge 0$ and $\lm\in G_{n}$,
\[
  F_{n}\big|_{G_{n}^{+}}(\lm) = - F_{n}\big|_{G_{n}^{-}}(\lm).
\]
Consequently, $F_{n}^{2}$ is continuous and hence analytic on all of $U_{n}$.

(iv)
If $q$ is real-valued, then for any $\lm\in G_{n}$
\[
  F_{n}(\lm\pm \ii 0) = \int_{\lm_{n}^{-}}^{\lm} \om\Big|_{G_{n}^{\pm}}
  = \pm \int_{\lm_{n}^{-}}^{\lm} \frac{(-1)^{n}\dDl}{\sqrt[+]{\Dl^{2}-4}} \dlm
  = \pm \arccosh \frac{(-1)^{n}\Dl(\lm)}{2}.
\]

(v)
At the zero potential, $\om(\lm,0) = 1/\bigl(2\ii\sqrt[+]{\lm}\bigr)\dlm$ which follows directly from the product representation \eqref{om} of $\om$. As $\lm_{0}^{+} = 0$ one gets $F_{0}(\lm,0) = \ii\sqrt[+]{\lm}$.\qed
\end{proof}

Since $F_{0} \equiv F_{n} - \ii n\pi$, we drop the subscript of $F_{0}$ in the sequel to simplify notation and denote
\begin{equation}
  F(\lm) \equiv F_{0}(\lm) = \int_{\lm_{0}^{+}}\om.
\end{equation}

\begin{rem*}
By exactly the same arguments one can show that $\om$ and $F_{n}$, $n\ge 0$ satisfy the properties stated in Lemma~\ref{w-closed} \& \ref{F-prop} when $\Ws$ is extended to the open neighborhood of $\Hs^{-1}_0$ within $\Hs^{-1}_{0,\C}$ constructed in \cite{Kappeler:2005fb}, and $\Ws_q$ is chosen as an open neighborhood of $q$ within $\Hs^{-1}_{0,\C}$ such that a common set of isolating neighborhoods exists. However, we will not make use of this fact.
\end{rem*}

If $q$ is real-valued, then we can integrate by parts and subsequently shrink the contour of integration to the interval $G_{n}$ to obtain
\begin{equation}
\label{J-F}
  J_{n,m}
   = -\frac{1}{\pi}\int_{\Gm_{n}} \lm^{m}F(\lm)\,\dlm
    = \frac{2}{\pi}\int_{G_{n}} \lm^{m}f_n(\lm)\,\dlm.
\end{equation}
Thus for real-valued potentials all action variables $J_{n,m}$ are real, and those on even levels are nonnegative. Moreover, by the mean value theorem,
\begin{align}
  \label{zt-In-Jn}
  J_{n,m} = \zt_{n,m}^{m} I_{n},
\end{align}
for some $\zt_{n,m}\in G_{n}$. By~\eqref{p-id} the actions on level zero satisfy
\begin{align}
  \label{tf-0}
  \sum_{n\ge 1} (2n\pi)I_{n} = \frac{1}{2}\int_{\T} q^{2}\,\dx = H_{0}.
\end{align}

In the sequel we derive similar formulae expressing the actions on any level $m\ge 0$ in terms of Hamiltonians of the \kdv hierarchy. To this end, we need the following expansion of $F$ for real-valued finite-gap potentials, which follows from the expansion of the discriminant $\Dl$ obtained in Appendix~\ref{a:dl-exp}.

\begin{prop}
\label{F-exp}
Suppose $q$ is a real-valued finite-gap potential with $[q] = 0$ and let $\nu_{n} = (n+1/2)\pi$. Then for any $K\ge 0$
\[
  F(\nu_{n}^{2})
   = -\ii\nu_{n} + \ii \sum_{0\le k \le K} \frac{H_{k}}{4^{k+1}\nu_{n}^{2k+3}}
     + O(\nu_{n}^{-2K-5}),
\]
and
\[
  F^{2}(\nu_{n}^{2})
   = -\nu_{n}^{2} + 2\sum_{0\le k\le K} \frac{H_{k}}{4^{k+1}\nu_{n}^{2k+2}}
   - \sum_{0\le k \le K}\sum_{0\le l \le k} \frac{H_{k-l}H_{l}}{4^{k+2}\nu_{n}^{2k+6}}
    + O(\nu_{n}^{-2K-4}),
\]
as $n\to \infty$.\fish
\end{prop}

\begin{proof}
Since $q$ is real-valued, the function $(-1)^{n+1} \Dl(\lm)$, for any $n\ge 0$, is strictly increasing from $-2$ to $2$ on $[\lm_{n}^{+},\lm_{n+1}^{-}]$ and the canonical root for $\lm_n^+ < \lm < \lm_{n+1}^-$ is given by
\[
  \sqrt[c]{\Dl^2(\lm) - 4}
   = (-1)^{n+1} \ii \sqrt[+]{4 - \Dl^2(\lm)}.
\]
Furthermore one computes for $\lm_n^+ < \lm < \lm_{n+1}^-$
\begin{align*}
  \partial_\lm \left(-\ii \asin\left((-1)^{n+1}\frac{\Dl(\lm)}{2}\right) \right)
   &= \ii \frac{(-1)^{n}\Dl^\ld(\lm)/2}{\sqrt[+]{1-\Dl^2(\lm)/4}}
   = \frac{\Dl^\ld(\lm)}{\sqrt[c]{\Dl^2(\lm) - 4}}.
\end{align*}
Hence for any $\lm_n^+ \le \lm \le \lm_{n+1}^-$,
\begin{align*}
  F(\lm)
  &= F(\lm_n^+) + \int_{\lm_n^+}^{\lm} \frac{\Dl^\ld(\lm)}{\sqrt[c]{\Dl^2(\lm) - 4}} \, \dlm \\
  &= -\ii n \pi + \left[-\ii \asin\left((-1)^{n+1}\frac{\Dl(\lm)}{2}\right) \right]_{\lm_n^+}^\lm \\
  &= -\ii (n + 1/2) \pi - \ii\asin\left((-1)^{n+1}\frac{\Dl(\lm)}{2}\right).
\end{align*}
Since $q$ as a finite-gap potential is smooth, it follows from Lemma \ref{dl-exp} that with $\nu_{n} = (n+1/2)\pi$ for any $N\ge 1$
\begin{equation}
  \label{dl-as}
  \frac{\Dl(\nu_{n}^2)}{2} = \cos \Th_N(\nu_{n}) + O(\nu_{n}^{-N-1}), \qquad n\to \infty,
\end{equation}
where
\begin{equation}
  \label{Th-as}
  \Th_N(\nu_{n}) = \nu_{n} - \sum_{3\le 2k+3\le N} \frac{H_{k}}{4^{k+1}\nu_{n}^{2k+3}}
                 = \nu_{n} + O(\nu_{n}^{-1}), \qquad n\to \infty.
\end{equation}
Using that $\partial_z \asin(z) = {1}/{\sqrt[+]{1-z^2}}$ and taking into account the estimates \eqref{dl-as}-\eqref{Th-as} one gets by the mean value theorem
\[
  \abs*{\asin\p*{ (-1)^{n+1}\frac{\Dl(\nu_{n}^{2})}{2} }
      - \asin\p*{ (-1)^{n+1} \cos \Th_N(\nu_{n}) }} = O(\nu_{n}^{-N-1}),
\]
and hence
\[
  F(\nu_{n}^{2}) = -\ii \nu_{n} - \ii\asin\p*{ (-1)^{n+1} \cos \Th_N(\nu_{n}) } + O(\nu_{n}^{-N-1}),
  \qquad n\to\infty.
\]

Finally, writing $(-1)^{n+1} = -\sin \nu_n$ one gets by the addition theorem for the sine
\[
  (-1)^{n+1} \cos \Th_N(\nu_{n}) = \sin\p*{ \Th_N(\nu_{n}) - \nu_n },
\]
and hence in view of~\eqref{Th-as}
\[
  \asin\p*{ (-1)^{n+1} \cos \Th_N(\nu_{n}) } = \Th_N(\nu_{n}) - \nu_{n}.
\]
This gives the claimed expansion for $F(\nu_{n}^{2})$. The claim for $F^{2}(\nu_{n}^{2})$ then follows by a straightforward computation.~\qed
\end{proof}

We are now in the position to prove the trace formula on an arbitrary level $m\ge 0$. The proof appears as a straightforward generalization from the case $m=0$ found in \cite[Appendix~E]{Kappeler:2003up}.

\begin{lem}[Trace formula]
\label{tf}
Suppose $m\ge 0$ and $q\in \Hs_0^m$, then
\begin{align}
  \label{tf-m}
  \sum_{n\ge 1} (2n\pi)J_{n,m}
   = \frac{1}{4^m}\H_m - \frac{2}{4^{m}}\sum_{0\le k\le m-2} \H_{m-2-k}\H_k,
\end{align}
where an empty sum denotes zero, and $\H_k$ denotes the $k$th Hamiltonian in the
\kdv hierarchy.~\fish
\end{lem}

\begin{proof}
Note that $\sum_{n\ge 1} n^{2m}\abs{\gm_n}^2 < \infty$ uniformly on bounded subsets of $\Hs_{0,\C}^{m}$ -- see \cite{Marcenko:1975wf,Kappeler:2001hsa} and also Appendix~\ref{s:appendix}. Moreover,
\[
  (2n\pi)J_{n,m} = O(n^{2m+1}I_{n}) = O(n^{2m}\gm_{n}^{2})
\]
uniformly on bounded subsets of $\Wp\cap \Hs_{0,\C}^{m}$ -- see~\cite{Kappeler:2003up}. Thus the sum of the actions on level $m$ converges locally uniformly to a real-analytic function on $\Hs^{m}_{0}$. Since also the right hand side of \eqref{tf-m} is real-analytic on $\Hs_{0}^{m}$, it suffices to prove the claim on the dense subset of real-valued finite-gap potentials.

Suppose $q$ is a real-valued finite-gap potential, then there exists $N\ge 1$ such that $\gm_{n}=0$ for $n > N$. Let $C_{r}$ denote a circle around the origin of sufficiently large radius $r$ so that all gaps $G_{n}$, $1\le n\le N$ are enclosed. Then $F^{2} = F_{0}^{2}$ is analytic outside $C_{r}$ by Lemma~\ref{F-prop} and by contour deformation
\[
  \int_{C_{r}} \lm^{m}F^{2}(\lm)\,\dlm = \sum_{1\le n\le N} \int_{\Gm_{n}} \lm^{m} F^{2}(\lm)\,\dlm.
\]
Since $F_{n}^{2}$ is analytic on $U_{n}$, the even powers of $F_{n}$ in the expansion of $F^{2}= (F_{n}-\ii n\pi)^{2}$ do not contribute to the contour integral, thus
\[
   \frac{1}{\ii\pi}\int_{\Gm_{n}} \lm^{m} F^{2}(\lm)\,\dlm =
  -\frac{2}{\ii\pi}\int_{\Gm_{n}} \lm^{m} (\ii n\pi )F_{n}(\lm)\,\dlm
   = (2 n\pi) J_{n,m}.
\]
On the other hand, according to Proposition~\ref{F-exp} for $\lm_{n} = (n+1/2)^{2}\pi^{2}$
\begin{equation}
  \label{F2-exp}
  F^{2}(\lm_{n})
   = -\lm_{n} + \sum_{0\le k\le m} \frac{2H_{k}}{4^{k+1}\lm_{n}^{k+1}}
   - \sum_{\atop{0\le k \le m}{0\le l \le k}} \frac{H_{k-l}H_{l}}{4^{k+2}\lm_{n}^{k+3}}
    + O(\lm_{n}^{-m-2}),
\end{equation}
as $n\to \infty$. One infers directly from~\eqref{om} that for a finite-gap potential, $\om = O(1/\sqrt{\lm})$ and hence $F^{2}(\lm) = O(\lm)$ as $\abs{\lm}\to \infty$. Therefore, $F^{2}(\lm)$ is meromorphic at infinity and we obtain from~\eqref{F2-exp} using Cauchy's Theorem
\begin{align*}
  \frac{1}{\ii\pi}\int_{C_{r}} \lm^{m}F^{2}(\lm)\,\dlm
  &= \frac{1}{4^{m}}H_{m} - \frac{2}{4^{m}}\sum_{0\le k \le m-2} H_{m-2-k}H_{k}.
\end{align*}
This proves the trace formula.\qed
\end{proof}

As an immediate consequence we obtain for $q\in\Hs_{0}^{m}$
\[
  4^{m}\sum_{n\ge 1} (2n\pi)J_{n,m} = \frac{1}{2} \int_{\T} (\partial_{x}^{m}q)^{2} \dx + \dotsb,
\]
where $\dotsb$ comprises only lower order derivatives of $q$. In Sections~\ref{s:act-est} and \ref{s:sob-est} this identity is used to compare the sum $\sum_{n\ge 1} (2n\pi) J_{n,m}$ of the actions on level $m$ with the corresponding Sobolev norm of the potential. To proceed, we need an estimate of the action $J_{n,m}$ on level $m\ge 0$ in terms of the weighted action $(n\pi)^{2m}I_{n}$.

\begin{prop}
\label{InJn}
Suppose $q\in \Hs_{0}^{0}$, then for $n\ge 4\n{q}_0$ and $m\ge 0$,
\[
  2^{-m}(n\pi)^{2m}I_{n} \le J_{n,m} \le 2^{m}(n\pi)^{2m}I_{n},
\]
while the remaining actions for $1\le n < 4\n{q}_{0}$ satisfy
\[
-5^{m}\n{q}_{0}^{2m} I_n \le J_{n,m} \le (256)^{m}\n{q}_{0}^{2m}I_{n}.\fish
\]
\end{prop}

\begin{proof}
Recall from \eqref{zt-In-Jn} that $J_{n,m} = \zt_{n,m}^{m}I_{n}$ with $\zt_{n,m}\in G_{n}$. With Proposition~\ref{ev-as} from the appendix we find for $n\ge 4\n{q}_{0}$,
\[
  \abs{\zt_{n,m}-(n\pi)^{2}} \le 4\n{q}_{0},
\]
while for $1\le n < 4\n{q}_{0}$,
\[
  -5\n{q}_{0}^2 \le (1+\n{q}_{0})\n{q}_{0} \le \zt_{n,m} \le 256\n{q}_{0}^{2}.\qed
\]
\end{proof}

% ===============================================================================================
\section{Estimating the actions}
\label{s:act-est}

In this section we obtain an estimate of all weighted $\ell^{1}$-norms $\n{I(q)}_{\ell^{1}_{2m+1}}$, $m\ge 1$, in terms of the Sobolev norms $\n{q}_{m}$ of $q$. This will be done in two steps. First, we estimate $\n{I(q)}_{\ell^{1}_{2m+1}}$ in terms of the sum $\sum_{n\ge 1} (2n\pi)J_{n,m}$ of the actions on level $m$ and a remainder depending solely on the $L^{2}$-norm $\n{q}_{0}$. Second, we express $\sum_{n\ge 1} (2n\pi)J_{n,m}$ through Hamiltonians of the \kdv hierarchy using the trace formula. The polynomial structure of these Hamiltonians then allows us to prove the claimed estimate.

\begin{lem}
\label{iqj-est}
Uniformly for all $q\in \Hs_{0}^{m}$ with $m\ge 0$,
\[
  \n{I(q)}_{\ell^{1}_{2m+1}}
  \le
  8^m\sum_{n\ge 1}(2n\pi) J_{n,m} +
  (8\pi)^{2m}\n{q}_{0}^{2m+2}.\fish
\]
\end{lem}

\begin{proof}
Clearly, by Proposition~\ref{InJn}
\[
  \sum_{n\ge 4\n{q}_0} (2n\pi)^{2m+1}I_{n}
   \le 8^{m}\sum_{n\ge 4\n{q}_0} (2n\pi)J_{n,m}.
\]
For $n < 4\n{q}_0$ the actions $J_{n,m}$ may become negative, albeit, with the lower bound $-5^{m}\n{q}_{0}^{2m} I_n \le J_{n,m}$, hence
\[
  \sum_{n \ge 4\n{q}_0} (2n\pi)J_{n,m}
  \le
  \sum_{n \ge 1} (2n\pi)J_{n,m}
  +
  5^{m}\n{q}_0^{2m}\sum_{1\le n < 4\n{q}_0} (2n\pi)I_{n}.
\]
The remaining actions may then be estimated by
\begin{align*}
  \sum_{1\le n < 4\n{q}_0} (2n\pi)^{2m+1}I_{n}
   &\le (8\pi)^{2m}\n{q}_{0}^{2m}\sum_{1\le n < 4\n{q}_0} (2n\pi) I_{n}.
\end{align*}
Both estimates together with $2\n{I(q)}_{\ell^{1}_{1}} = \n{q}_{0}^{2}$ give the claim.\qed
\end{proof}

\begin{proof}[Proof of Theorem~\ref{act-sob-est} (i) for $m=1$.]
Lemma~\ref{iqj-est} and the trace formula \eqref{tf-m} yield
\[
  \n{I(q)}_{\ell^{1}_{3}}
   \le 2\H_1(q) + 8^2\pi^2\n{q}_0^4.
\]
Using the $L^\infty$-estimate $\n{q}_{L^{\infty}} \le \n{q_{x}}_{0}$ on $\Hs_{0}^{1}$, we obtain for the Hamiltonian
\[
  \abs{\H_1(q)}
    \le \frac{1}{2}\n{q_x}_0^2 + \n{q}_{L^{\infty}}\n{q}_0^2
    \le \n{q}_1^2 + \frac{1}{2}\n{q}_0^4.
\]
Thus we arrive at
\[
  \n{I(q)}_{\ell^{1}_{3}} \le 2\n{q}_{1}^{2} + 16^2\pi^2\n{q}_{0}^{4}.\qed
\]
\end{proof}

To proceed with the general case $m\ge 2$, we need an estimate of the Hamiltonian $\H_{m}(q)$ for $m\ge 2$. To this end, we recall some well known facts about the Hamiltonians in the \kdv hierarchy -- see for example \cite[Theorem~1]{McKean:1975fz}.

\begin{lem}
\label{h-form}
The $m$th Hamiltonian in the \kdv hierarchy has the form
\[
  \H_m(q)  = \frac{1}{2}\int_\T
             \left[(\partial_{x}^{m}q)^2 + p_m(q,\partial_{x}q,\ldots,\partial_{x}^{m-1}q)\right]\,\dx,
\]
with $p_m$ being a homogeneous polynomial of degree $m+2$, without constant term, where $q$ counts as 1 degree and differentiation as 1/2 degree.\fish
\end{lem}

Consequently, each monomial $\mathfrak{p}$ of $p_{m}$ may be estimated by
\[
  \abs{\mathfrak{p}}
   \le c_{\mathfrak{p}} \abs{q}^{\mu_{0}}\abs{q_{x}}^{\mu_{1}}\dotsm\abs{q_{(m-1)}}^{\mu_{m-1}},
\]
with some positive constant $c_{\mathfrak{p}}$ and integers $\mu_{0},\dotsc,\mu_{m-1}$, where we denote $q_{(m)}\defl \partial_{x}^{m}q$ to simplify notation. It turns out to be convenient to use exponents which are multiples of two, that is
\[
  \abs{\mathfrak{p}} \le c_{\mathfrak{p}} \abs{q}^{2\sg_{0}}\abs{q_{x}}^{2\sg_{1}}
  \dotsm\abs{q_{(m-1)}}^{2\sg_{m-1}},
\]
with $\sg_{i}\in \Z_{\ge 0}/2$. Since $\mathfrak{p}$ is homogenous of degree $m+2$,
\begin{align}
  \label{p-deg}
  \sum_{0\le i\le m-1} (2+i)\sg_{i} = m+2.
\end{align}
Thus, we obtain the estimate
\begin{align}
 \abs{p_{m}} \le P_{m},\qquad
  \label{p-est}
  P_{m} \defl
   \sum_{\sg\in\Ic_{m}} c_{\sg} \abs{q}^{2\sg_{0}}\abs{q_{x}}^{2\sg_{1}}
   \dotsm\abs{q_{(m-1)}}^{2\sg_{m-1}},
\end{align}
with positive reals $c_{\sg}$ and $\Ic_{m}\subset (\Z_{\ge0}/2)^{m}$ being the set of all multi-indices satisfying the constraint \eqref{p-deg}.
The majorant $P_{m}$ allows us to obtain detailed estimates of $p_{m}$ and therefore $\H_{m}$.

\begin{proof}[Proof of Theorem~\ref{act-sob-est} (i) for $m\ge 2$.]
For $\sg\in\Ic_{m}$ we have by \eqref{p-deg}
\[
  m+2 \ge (m+1)\sg_{m-1},
\]
hence $P_{m}$ is at most quadratic in $q_{(m-1)}$. Together with the $L^{1}$-estimate $\n{q_{(m-1)}}_{L^1} \le \n{q_{(m-1)}}_0$ we obtain for the derivative of highest order
\[
  \int_\T \abs{q_{(m-1)}^{2\sg_{m-1}}}\,\dx
   \le \n{q_{(m-1)}}_0^{2\sg_{m-1}}.
\]
Estimating the remaining factors using the $L^{\infty}$-estimate
\[
  \abs{q_{(j)}^{2\sg_{j}}}
   \le
  \n{q_{(j)}}_{L^\infty}^{2\sg_{j}}
   \le
  \n{q}_{j+1}^{2\sg_{j}},
\]
then gives
\begin{align*}
  \int_\T \abs{p_{m}} \dx \le \int_\T P_{m} \dx
   \le \sum_{\Ic_{m}} c_{\sg} \n{q}_{m-1}^{2\sg_0 + \dotsb + 2\sg_{m-1}}.
\end{align*}
Since $2\sg_0 + \dotsb + 2\sg_{m-1} \le m+2$, we obtain for the Hamiltonian
\[
  \abs{\H_{m}(q)}
   \le
  \frac{1}{2}\n{q}_{m}^{2} +  b_m^2(1+\n{q}_{m-1}^{m})\n{q}_{m-1}^{2},
\]
with some absolute constant $b_m$.

Recall from Lemma~\ref{iqj-est} that
\[
  \n{I(q)}_{\ell_{2m+1}^{1}}
  \le
  8^m\sum_{n\ge 1} (2n\pi)J_{n,m}
  +
  (8\pi)^{2m}\n{q}_0^{2m+2},
\]
and combining the trace formula \eqref{tf-m} with the preceding estimate of the Hamiltonians yields
\begin{align*}
  4^{m}\abs[\Big]{\sum_{n\ge 1} (2n\pi)J_{n,m}}
   &\le
  \abs{H_m} + 2\sum_{k=0}^{m-2}\abs{H_{m-2-k}H_k}\\
   &\le
  \frac{1}{2}\n{q}_{m}^{2}
  +
  c_m(1+\n{q}_{m-1})^m\n{q}_{m-1}^2.
\end{align*}
This proves estimate (i) of Theorem~\ref{act-sob-est}.\qed
\end{proof}

% ===============================================================================================
\section{Estimating the Sobolev norms}
\label{s:sob-est}

We now turn to the problem of controlling the Sobolev norm $\n{q}_{m}$ of the potential
in terms of weighted norms of its actions. As a starting point we use the identity
\begin{align}
  \label{q-H}
  \frac{1}{2}\n{\partial_{x}^{m}q}^2
   =
   \H_m - \frac{1}{2}\int_\T  p_m(q,\ldots,\partial_{x}^{m-1}q)\,\dx,
\end{align}
inferred from Lemma~\ref{h-form}, and proceed in two steps. First, we estimate the sum $\sum_{n\ge 1} (2n\pi)J_{n,m}$ of the actions on level $m$ in terms of $\n{I(q)}_{\ell^{1}_{2m+1}}$ and a remainder depending solely on $\n{q}_0^{2} = 2\n{I(q)}_{\ell^{1}_{1}}$. From the trace formula we then obtain a bound of $\H_m$ in terms of the weighted action norms. In the second step we use the polynomial structure of $p_m$ to obtain an estimate of the $p_m$-integral involving at most $\n{q_{(m-1)}}^{2}$ but not $\n{q_{(m)}}^{2}$. Here, we again use the notation $q_{(m)} = \partial_{x}^{m}q$. An inductive argument then gives the claim.

\begin{lem}
\label{J-I-est}
On $\Hs_{0}^{m}$ for any $m\ge 1$,
\[
  \sum_{n\ge 1} (2n\pi) \abs{J_{n,m}}
   \le
   2^{-m}\n{I(q)}_{\ell^{1}_{2m+1}}
        + 2^{9m}\n{I(q)}_{\ell_{1}^{1}}^{m+1}.\fish
\]
\end{lem}

\begin{proof}
By Proposition~\ref{InJn}, for $n\ge 4\n{q}_0$,
\[
  \abs{J_{n,m}} \le 2^{m}(n\pi)^{2m}I_n,
\]
while for $1\le n < 4\n{q}_0$,
\[
  \abs{J_{n,m}} \le (256)^{m}\n{q}_{0}^{2m} I_{n}.
\]
Since $\n{q}_{0}^{2} = 2\n{I(q)}_{\ell^{1}_{1}}$ by the trace formula \eqref{tf-0}, the claim follows.\qed
\end{proof}

\begin{thm}
\label{h-i-est}
For any $m\ge 1$ there exists an absolute constant $c_{m}$ such that
\[
  \abs{H_m(q)}
   \le 2^{m}\n{I(q)}_{\ell^{1}_{2m+1}}
        + c_m(1+\n{I(q)}_{\ell^{1}_{2m-1}})^{m}\n{I(q)}_{\ell^{1}_{2m-1}},
\]
for all $q\in \Hs_{0}^{m}$.\qed
\end{thm}

\begin{proof}
We have $H_{0} = \n{I(q)}_{\ell_{1}^{1}}$ by \eqref{tf-0}, and further by \eqref{tf-m} for any $m\ge 1$
\[
  \abs{\H_m}
   \le
  4^m\sum_{n\ge 1}(2n\pi)\abs{J_{n,m}} + 2\sum_{k=0}^{m-2}\abs{H_{m-2-k}H_k}.
\]
Lemma~\ref{J-I-est} gives an estimate of the first summand, while we inductively obtain for the second
\begin{align*}
  \sum_{k=0}^{m-2} \abs{H_{m-2-k}H_k}
   &\le
  \sum_{k=0}^{m-2}
  c_{m-2-k}c_{k}(1+\n{I}_{\ell^{1}_{2m-5}})^{m-2}\n{I}_{\ell^{1}_{2m-3}}^{2}\\
   &\le
   \tilde c_{m}(1+\n{I}_{\ell^{1}_{2m-5}})^{m-1}\n{I}_{\ell^{1}_{2m-3}}.
\end{align*}
This proves the claim.\qed
\end{proof}

\begin{proof}[Proof of Theorem~\ref{act-sob-est} (ii).] We begin with the case $m=1$. Clearly,
\[
  \frac{1}{2}\n{q_x}_0^2 = \H_{1}(q) - \int q^{3}\,\dx.
\]
With the $L^{\infty}$-estimate $\n{q}_{L^{\infty}} \le \n{q_{x}}_{0}$ on $\Hs_{0}^{1}$ we find
\[
  \abs*{\int_\T  q^3\,\dx} \le \n{q_{x}}_{0}\n{q}_{0}^{2}
  						  \le \frac{1}{4}\n{q_{x}}_{0}^{2} + \n{q}_{0}^{4},
\]
and, together with $\n{q}_{0}^{2} = 2\n{I(q)}_{\ell_{1}^{1}}$, we therefore obtain
\[
  \frac{1}{4}\n{q_{x}}_0^2
   \le
  H_1(q) + 4\n{I(q)}_{\ell^{1}_{1}}^2.
\]
Finally, by the preceding Theorem,
\[
  \H_{1}(q) \le 2\n{I(q)}_{\ell^{1}_{3}} + c_{1}(1+\n{I(q)}_{\ell^{1}_{1}})\n{I(q)}_{\ell^{1}_{1}},
\]
such that
\[
  \n{q_{x}}_{0}^{2} \le 8\n{I(q)}_{\ell^{1}_{3}}
   + \tilde c_{1}(1+\n{I(q)}_{\ell^{1}_{1}})\n{I(q)}_{\ell^{1}_{1}}.
\]

To proceed with the general case $m\ge 2$, we recall from \eqref{q-H} that
\[
  \frac{1}{2}\n{q_{(m)}}^2
   =
   \H_m - \frac{1}{2}\int_\T  p_m(q,\ldots,q_{(m-1)})\,\dx.
\]
Theorem~\ref{h-i-est} provides an estimate of $H_{m}$, while for $p_{m}$ we have by~\eqref{p-est}
\[
  \abs{p_{m}} \le P_{m}
   = \sum_{\sg\in\Ic_{m}} c_{\sg}
     	\abs{q}^{2\sg_{0}}\abs{q_{x}}^{2\sg_{1}}\dotsm\abs{q_{(m-1)}}^{2\sg_{m-1}},
\]
where the majorant $P_{m}$ only contains derivatives of $q$ up to order $m-1$ and the multi-indices $\sg$ satisfy the constraint \eqref{p-deg}. We use this to prove by induction for any $m\ge 2$
\begin{align*}
  \frac{1}{2}\n{q_{(m)}}_0^2
   \le \n{I}_{\ell^{1}_{2m+1}} + d_m^2(1+\n{I}_{\ell^{1}_{2m-1}})^{m}\n{I}_{\ell^{1}_{2m-1}}.
\end{align*}
Clearly, this holds true for $m=0,1$.

Consider the inductive step $m-1\mapsto m$. Recall from the proof of Theorem~\ref{act-sob-est} (i) that $P_{m}$ is at most quadratic in $\abs{q_{(m-1)}}$ for $m\ge 2$. Thus, we have the following expansion
\begin{align*}
  P_{m} &= P_{m;2}(q,\ldots,q_{(m-2)})\abs{q_{(m-1)}}^2\\
  &\qquad + P_{m;1}(q,\ldots,q_{(m-2)})\abs{q_{(m-1)}}\\
  &\qquad + P_{m;0}(q,\ldots,q_{(m-2)}),
\end{align*}
where for $0\le k\le 2$ the entity $P_{m;k}\abs{q_{(m-1)}}^k$ incorporates those multi-indices $\sg\in\Ic_{m}$ with $2\sg_{m-1} = k$. In particular, $P_{m;2} = c\abs{q}$ with some $c\ge 0$. These three terms will be estimated separately.

We begin with $P_{m;2}\abs{q_{(m-1)}}^{2} = c\abs{q}\abs{q_{(m-1)}}^{2}$. To avoid a flood a constants, we write $a\lew b$ if $a\le c\cdot b$ with an absolute constant $c$ which is independent of $q$ but may depend on $m$. Using the $L^{\infty}$-estimate and the induction hypothesis for $m-1$, we get
\begin{align*}
  \int_\T P_{m;2}\abs{q_{(m-1)}}^2\,\dx
   &\lew
  \n{q}_{L^{\infty}}\n{q_{(m-1)}}_{0}^2\\
   &\lew
  \n{q_{x}}_0\n{q_{(m-1)}}_{0}^2\\
   &\lew
  (1+\n{I}_{\ell^{1}_{2m-1}})^{m}\n{I}_{\ell^{1}_{2m-1}}.
\end{align*}

Second, consider the term $P_{m;1}\abs{q_{(m-1)}}$ incorporating all multi-indices with $2\sg_{m-1} = 1$. Since $\sum_{0\le i\le m-1} (2+i)\sg_i = m+2$, there has to be a nonzero $\sg_{j}$ for some $0\le j\le m-2$.
By Cauchy-Schwarz and the $L^{\infty}$-estimate
\begin{align*}
  \int_\T \abs{q_{(j)}}^{2\sg_j}\abs{q_{(m-1)}}\,\dx
   &\lew
  \n{q_{(j+1)}}_{0}^{2\sg_j-1}\n{q_{(j)}}_0\n{q_{(m-1)}}_0\\
   &\lew
  (1+\n{I}_{\ell^{1}_{2m-1}})^{(j+1)(\sg_{j}-1/2) + j/2 + (m-1)/2}
  \n{I}_{\ell^{1}_{2m-1}}^{\sg_j + 1/2}\\
   &\lew
  (1+\n{I}_{\ell^{1}_{2m-1}})^{(j+1)\sg_{j} + m\sg_{m-1} -1}
  \n{I}_{\ell^{1}_{2m-1}}^{\sg_j + \sg_{m-1}}.
\end{align*}
For the remaining factors the $L^{\infty}$-estimate gives
\[
  \abs{q_{(k)}}^{2\sg_k}
   \le
  \n{q_{(k+1)}}_{0}^{2\sg_k}
   \lew
  (1+\n{I}_{\ell^{1}_{2m-2}})^{(k+1)\sg_k}\n{I}_{\ell^{1}_{2m-2}}^{\sg_k}.
\]
Both estimates together yield
\begin{align*}
  &\int_\T \abs{q}^{2\sg_{0}}\dotsm\abs{q_{(m-1)}}^{2\sg_{m-1}}\,\dx\\
  &\qquad\lew
  (1+\n{I}_{\ell^{1}_{2m-1}})^{\sg_{0} + 2\sg_{1} + \dotsb + m\sg_{m-1} - 1} \n{I}_{\ell^{1}_{2m-2}}^{\sg_0 + \dotsb +
  \sg_{m-1}}\\
  &\qquad\lew
  (1+\n{I}_{\ell^{1}_{2m-1}})^{(2+0)\sg_0 + \dotsb +
  (2+(m-1))\sg_{m-1} - 2}
  \n{I}_{\ell^{1}_{2m-2}}.
\end{align*}
As $(2+0)\sg_0 + \dotsb + (2+(m-1))\sg_{m-1} - 2 = m$ by \eqref{p-deg}, we conclude
\[
  \int_\T P_{m;1}\abs{q_{(m-1)}}\,\dx
   \lew (1+\n{I}_{\ell^{1}_{2m-2}})^{m}\n{I}_{\ell^{1}_{2m-2}}.
\]

It remains to estimate the term $P_{m;0}$ which incorporates those $\sg$ with $\sg_{m-1} = 0$. First consider the case that $\sg_{i},\sg_{j}\neq 0$ for some $0\le i < j\le m-2$. By Cauchy-Schwarz,
\begin{align*}
  \int_\T \abs{q_{(i)}}^{2\sg_i}\abs{q_{(j)}}^{2\sg_j}\,\dx
   &\lew
  \n{q_{(i+1)}}_{0}^{2\sg_i-1}\n{q_{(j+1)}}_{0}^{2\sg_j-1}\n{q_{(i)}}_0\n{q_{(j)}}_0\\
   &\lew
  (1+\n{I}_{\ell^{1}_{2m-1}})^{(i+1)\sg_{i} + (j+1)\sg_{j}-1}
  \n{I}_{\ell^{1}_{2m-1}}^{\sg_i + \sg_j},
\end{align*}
while the remaining factors are estimated as usual,
\[
  \abs{q_{(k)}}^{2\sg_k}
   \lew
  \n{q_{(k+1)}}_{0}^{2\sg_k}
   \lew
  (1+\n{I}_{\ell^{1}_{2m-2}})^{(k+1)\sg_k}\n{I}_{\ell^{1}_{2m-2}}^{\sg_k},
\]
such that
\begin{align*}
  &\int_\T \abs{q}^{2\sg_{0}}\dotsm\abs{q_{(m-2)}}^{2\sg_{m-2}}\,\dx \\
   &\qquad \lew
   (1+\n{I}_{\ell^{1}_{2m-1}})^{\sg_0 + 2\sg_1 + \dotsb + (m-1)\sg_{m-2}-1}
      \n{I}_{\ell^{1}_{2m-1}}^{\sg_0 + \dotsb + \sg_{m-2}}\\
   &\qquad\lew
  (1+\n{I}_{\ell^{1}_{2m-1}})^{(2+0)\sg_0 + \dotsb + (2+(m-2))\sg_{m-2} - 2}
     \n{I}_{\ell^{1}_{2m-1}}.
\end{align*}
As before $(2+0)\sg_0 + \dotsb + (2+(m-2))\sg_{m-2} - 2 = m$ by~\eqref{p-deg}. Finally, consider the case where all $\sg_{j}$ except for $\sg_{i}$ with $0\le i\le m-2$ vanish. Then $(2+i)\sg_{i} = m+2$, and consequently
\begin{align*}
  \int_\T \abs{q_{(i)}}^{2\sg_i}\,\dx
   &\lew
  \n{q_{(i+1)}}_{0}^{2\sg_i-2}\n{q_{(i)}}_0^{2}\\
   &\lew
  (1+\n{I}_{\ell^{1}_{2m-1}})^{(i+1)(\sg_i-1)+i}\n{I}_{\ell^{1}_{\ell^{1}_{2m-1}}}^{\sg_i}\\
   &\lew
  (1+\n{I}_{\ell^{1}_{2m-1}})^{m}\n{I}_{\ell^{1}_{\ell^{1}_{2m-1}}}.
\end{align*}

Altogether we thus obtain
\[
  \int_\T P_{m}\,\dx
   \lew
  (1+\n{I}_{\ell^{1}_{2m-1}})^{m}\n{I}_{\ell^{1}_{2m-1}},
\]
which completes the induction step and proves Theorem~\ref{act-sob-est} (ii).\qed
\end{proof}

% =================================================================================================
\section{Estimating the Actions in Weighted Sobolev Spaces}
\label{s:act-west}

The case of estimating the actions in arbitrary weighted Sobolev spaces $\Hs_{0}^{w}$ differs significantly from the case of integer Sobolev spaces $\Hs_{0}^{m}$ since for an arbitrary weight $w$ there is no identity known to exist relating $\n{q}_{w}$ to Hamiltonians of the \kdv hierarchy. Albeit, even in the case of weighted Sobolev spaces, the regularity properties of $q$ are well known to be closely related to the decay properties of the  gap lengths $\gm_{n}(q)$ -- see e.g. \cite{Poschel:2011iua,Djakov:2006ba,Kappeler:1999er} and Appendix~\ref{s:appendix}. Moreover, the asymptotic relation
\begin{align}
  \label{In-gmn}
  \frac{8n\pi I_{n}}{\gm_{n}^{2}} = 1 + O\left(\frac{\log n}{n}\right)
\end{align}
is known to hold locally uniformly on $\Hs_{0}^{0}$ -- see \cite{Kappeler:2003up}. In this section we obtain a quantitative version of \eqref{In-gmn} which is uniform in $\n{q}_{0}$ on all of $\Hs_{0}^{0}$. This together with the estimates of the gap lengths given in the appendix allows us to prove Theorem~\ref{act-west}.

For $q\in \Hs_{0}^{0}$ we recall from \eqref{Jn-om} that
\[
  I_{n} = \frac{1}{\pi}\int_{\Gm_{n}} \lm \om = -\frac{1}{\pi}\int_{\Gm_{n}} (\lm_{n}^{\ld}-\lm)\om.
\]
Here the latter identity follows from the closedness of $\om$ around the gap. In the case $I_{n}\neq 0$, or equivalently $\gm_{n}\neq 0$, we shrink the contour $\Gm_{n}$ to the straight line $[\lm_{n}^{-},\lm_{n}^{+}]$ and insert the product representation \eqref{om} of $\om$, to obtain
\[
  I_{n} = \frac{1}{\pi}\int_{\lm_{n}^{-}}^{\lm_{n}^{+}}
  \frac{(\lm_{n}^{\ld}-\lm)^{2}\,\chi_{n}(\lm)}{\sqrt[+]{\gm_{n}^{2}/4-(\tau_{n}-\lm)^{2}}}\,\dlm,
  \qquad
  \chi_{n}(\lm) = \frac{1}{\sqrt[+]{\lm-\lm_{0}^{+}}}
        \prod_{m\neq n} \frac{\lm_{m}^{\ld}-\lm}{\vs_{m}(\lm)}.
\]
Parametrizing the gap $G_{n} = [\lm_{n}^{-},\lm_{n}^{+}]$ by $\lm_{t} = \tau_{n} + t\gm_{n}/2$ gives
\[
  \frac{8n\pi I_{n}}{\gm_{n}^{2}} = \frac{2}{\pi}\int_{-1}^{1}
   \frac{(t-t_{n})^{2}}{\sqrt[+]{1-t^{2}}}(n\pi)\chi_{n}(\tau_{n}+t\gm_{n}/2) \,\dt,
\]
where we set $t_{n} = 2(\lm_{n}^{\ld}-\tau_{n})/\gm_{n}$. Since $\abs{\tau_{n}-\lm_{n}^{\ld}} \le \gm_{n}/2$ we conclude $\abs{t_{n}}\le 1$, and hence
\begin{align}
  \label{In-gmn-2}
  \frac{8n\pi I_{n}}{\gm_{n}^{2}} \le 3(n\pi)\max_{\lm\in G_{n}}\,\abs{\chi_{n}(\lm)}.
\end{align}
The following uniform estimate of $(n\pi)\abs{\chi_{n}}_{G_n}$ allows us to proof the desired quantitative version of~\eqref{In-gmn} in the sequel.

\begin{lem}
On $\Hs_{0}^{s}$ with $0\le s\le 1/2$ for any $n \ge 8\n{q}_{s}^{3/2-s}$,
\[
  (n\pi)\max_{\lm\in G_{n}}\,\abs{\chi_{n}(\lm)} \le 64(1+\n{q}_{s})^{3/2-s}.\fish
\]
\end{lem}

\begin{proof}
For $n\neq m$ with $n,m \ge 4\n{q}_{0}$ and $\lm\in G_{n}$ we have by Lemma~\ref{Sn-roots}
\begin{align*}
  \abs{\tau_{m}-\lm} &\ge \abs{n^{2}\pi^{2}-m^{2}\pi^{2}} - \abs{\tau_{m}-m^{2}\pi^{2}}
   - \abs{\lm-n^{2}\pi^{2}}\\
                     &\ge \pi^{2}\abs{n^{2}-m^{2}} - 8\n{q}_{0}
                      \ge 8\abs{n^{2}-m^{2}}.
\end{align*}
Further, $\abs{\gm_{m}} \le 6\n{q}_{0}$ by Proposition~\ref{gap-est}, hence
\[
  \dl_{m} \defl \frac{\abs{\gm_{m}}/2}{\abs{\tau_{m}-\lm}}
   \le \frac{3\n{q}_{0}}{8\abs{n^{2}-m^{2}}} \le 1/8.
\]
Since $\lm_{m}^{-}\le \lm_{m}^{\ld}\le \lm_{m}^{+}$ we have $\abs{\lm_{m}^{\ld}-\tau_{m}} \le \gm_{m}/2$ and thus
\begin{align*}
  \abs*{\frac{\lm_{m}^{\ld}-\lm}{\vs_{m}(\lm)}}
   \le \frac{\abs{\tau_{m}-\lm} + \abs{\gm_{m}}/2}{\abs{\tau_{m}-\lm} - \abs{\gm_{m}}/2}
   &\le 1 + \frac{1}{1 - \dl_{m}}\frac{\abs{\gm_{m}}}{\abs{\tau_{m}-\lm}}\\
   &\le 1 + \frac{\abs{\gm_{m}}}{4\abs{n^{2}-m^{2}}}.
\end{align*}
Suppose $N \ge 4\n{q}_{s}$, then by Proposition~\ref{gap-est}
\[
  \sum_{m\ge N} \lin{2m}^{2s}\abs{\gm_{m}}^{2}
   \le 9\n{q}_{s}^{2} + \frac{576}{N}\n{q}_{s}^{4},
\]
while on the other hand for any $n \ge N$
\[
  \sum_{\atop{m\neq n}{m \ge N}} \frac{1}{\abs{n^{2}-m^{2}}^{2}}
    \le \frac{2}{(n+N)^{2}}\sum_{m\ge 1}\frac{1}{m^{2}}
    \le \frac{1}{N^2}.
\]
Therefore, by Cauchy-Schwarz,
\begin{align*}
  \abs[\bigg]{\sum_{\atop{m\neq n}{m \ge N}}
    \frac{\gm_{m}}{4\abs{n^{2}-m^{2}}}}^{2}
      &\le \frac{9}{16N^{2+2s}}\n{q}_{s}^{2}
         + \frac{36}{N^{3+2s}}\n{q}_{s}^{4}.
\end{align*}
Clearly $4/(3+2s) \le 3/2-s$ for $0\le s\le 1/2$. So if we increase $N$ such that $N-1\le \max(4\n{q}_{s}^{3/2-s},4\n{q}_{s}) \le N$, or more conveniently $N-1 \le 8\n{q}_{s}^{3/2-s} \le N$, then by the standard estimates for infinite products,
\[
  \prod_{\atop{m\neq n}{m \ge N}} \abs*{\frac{\lm_{m}^{\ld}-\lm}{\vs_{m}(\lm)}}
  \le \exp\left(\sum_{\atop{m\neq n}{m \ge N}}
    \frac{\gm_{m}}{4\abs{n^{2}-m^{2}}}\right)
  \le \exp\left(1/2\right) \le 2.
\]

To estimate the remaining part of the product we note that by the ordering of the eigenvalues
\[
  \abs{\lm_{m}^{\ld}-\lm} \le \abs{\lm_{m-1}^{\pm}-\lm} ,\qquad 2\le m< N,
\]
and consequently (if $N\ge 2$)
\[
  \prod_{1\le m <  N} \abs*{\frac{\lm_{m}^{\ld}-\lm}{\vs_{m}(\lm)}}
  =   \abs*{\frac{\lm_{1}^{\ld}-\lm}{\vs_{N-1}(\lm)}}
      \prod_{2\le m <  N} \abs*{\frac{\lm_{m}^{\ld}-\lm}{\vs_{m-1}(\lm)}}
  \le \frac{\abs{\lm_{1}^{\ld}-\lm}}{\abs{\vs_{N-1}(\lm)}}.
\]
By Proposition~\ref{ev-as} and Lemma~\ref{Sn-roots} we have for $n\ge N$ and $\lm\in G_{n}$,
\[
  \abs{\lm-\lm_{N-1}^{\pm}} \ge (n^{2}-N^{2})\pi^{2} - 4\n{q}_{0} + 12N
                            \ge 9(n^{2}-N^{2}) + 11N,
\]
while on the other hand by Proposition~\ref{ev-as}
\[
  \abs{\lm_{1}^{\ld}-\lm}
    \le \abs{\lm_{0}^{+}-\lm}
    \le n^{2}\pi^{2} + 4\n{q}_{0} + (1+\n{q}_{0})\n{q}_{0}
  \le 12n^{2}.
\]
Both estimates together yield (if $N\ge 2$)
\[
  \prod_{1\le m <  N} \abs*{\frac{\lm_{m}^{\ld}-\lm}{\vs_{m}(\lm)}}
  \le \frac{12n^{2}}{11N + 9(n^{2}-N^{2})} \le 2N.
\]

Finally, since $\lm_{0}^{+} \le [q] = 0$, we find
\[
  \sqrt[+]{\lm-\lm_{0}^{+}} \ge \sqrt[+]{n^{2}\pi^{2}-4\n{q}_{0}} \ge n\pi/2,
\]
and consequently for any $n\ge N$
\[
  (n\pi)\abs{\chi_{n}(\lm)} \le
  \frac{n\pi}{\sqrt[+]{\lm-\lm_{0}^{+}}}\prod_{m\neq n}
  \abs*{\frac{\lm_{m}^{\ld}-\lm}{\vs_{m}(\lm)}} \le 8N \le 64(1+\n{q}_{s})^{3/2-s}.\qed
\]
\end{proof}

\begin{cor}
On $\Hs_{0}^{s}$ with $0\le s\le 1/2$ for any $n > 8\n{q}_{s}^{3/2-s}$,
\[
  (2n\pi)I_{n} \le 48(1+\n{q}_{s})^{3/2-s}\gm_{n}^{2}.\fish
\]
\end{cor}

\begin{proof}
If $\gm_{n}=0$, then $I_{n} = 0$ and the estimate clearly holds. If $\gm_{n}\neq 0$, then by \eqref{In-gmn-2} and the preceding lemma,
\[
  (8n\pi)I_{n}/\gm_{n}^{2} \le 3(n\pi)\abs{\chi_{n}}_{G_{n}}
  \le 192(1+\n{q}_{s})^{3/2-s}.\qed
\]
\end{proof}

\begin{proof}[Proof of Theorem~\ref{act-west}.]
Suppose $q\in \Hs_{0}^{w}$ with $w = \lin{n}^{s}v$ where $v\in \Ms$ and $0\le s\le 1/2$. Choose $N\ge 1$ such that $N\ge 8\n{q}_{w}^{3/2-s} > N-1$. Then by the preceding corollary
\[
  \sum_{n \ge N} (2n\pi)w_{2n}^{2} I_{n}
   \le 48(1+\n{q}_{s})^{3/2-s}\sum_{n \ge N} w_{2n}^{2}\gm_{n}^{2},
\]
and for the gap lengths we obtain with Proposition~\ref{gap-est}
\[
  \sum_{n \ge N} w_{2n}^{2}\gm_{n}^{2}
  \le 9\n{q}_{w}^{2} + 72\n{q}_{w}^{5/2+s}
  \le 72(1+\n{q}_{w}^{1/2+s})\n{q}_{w}^{2}.
\]
The remaining actions for $1\le n< N$ may be estimated by
\[
  \sum_{1\le n< N} w_{2n}^{2}(2n\pi) I_{n}
   \le w^{2}_{2N-2}\sum_{n\ge 1} (2n\pi)I_{n}
   \le \bigl(w(16\n{q}_{w}^{3/2-s})\bigr)^{2}\n{q}_{0}^{2}.
\]
Altogether, we thus find
\[
  \sum_{n\ge 1} w_{2n}^{2}(2n\pi)I_{n}
   \le \left(2^{12}(1+\n{q}_{w})^{2} + \bigl(w(16\n{q}_{w}^{3/2-s})\bigr)^{2}\right)\n{q}_{w}^{2}.\qed
\]
\end{proof}

\begin{appendix}

\theoremstyle{accentheader}
\theorembodyfont{\itshape}
\renewtheorem{prop}{Proposition}[section]
\renewtheorem{cor}[prop]{Corollary}

\section{Appendix - Spectral Theory}
\label{s:appendix}

In this appendix we review, for the convenience of the reader, the localization of the periodic spectrum of Hill's operator as well as an estimate of its gap lengths, which both are used in the various parts of this paper. We follow the exposition in \cite{Poschel:2011iua} -- see also
\cite{Kappeler:1999er,Kappeler:2001hsa,Djakov:2006ba}.
Consider the operator
\[
  L(q) = -\ddx + q,
\]
on the interval $[0,2]$ endowed with periodic boundary conditions and $q$ being a complex-valued, $1$-periodic $L^{2}$-potential with vanishing mean value, that is $q\in \Hs_{0,\C}^{0} = \Hs_{0}^{0}(\T,\C)$.

The spectrum of $L(q)$ for $q=0$ consists of $\lm_{0}^{+} = 0$ and the double eigenvalues $\lm_{n}^{+} = \lm_{n}^{-} = n^{2}\pi^{2}$, $n\ge 1$. For $q\in \Hs_{0,\C}^{0}$ arbitrary and $\lm$ sufficiently large, the equation $-f'' + qf = \lm f$ may be regarded as a perturbation of the free equation $-f'' = \lm f$, hence one can expect the eigenvalues to come asymptotically in pairs $\lm_{n}^{\pm}$ satisfying $\lm_{n}^{\pm}\sim n^{2}\pi^{2}$ as $n\to \infty$. The following localization of the eigenvalues is well known -- see e.g. \cite{Poschel:2011iua,Djakov:2006ba,Kappeler:1999er}.

\begin{prop}
\label{ev-as}
If $q\in \Hs_{0,\C}^{0}$, then for all $n\ge 4\n{q}_0$,
\[
  \abs{\lm_{n}^{\pm}-n^{2}\pi^{2}} \le 4\n{q}_{0}.
\]
The remaining eigenvalues for $4\n{q}_{0} > n$ satisfy
\[
  -(1+\n{q}_{0})\n{q}_{0} \le \Re\lm_{0}^{+} \le \Re \lm_{n}^{\pm} \le 256\n{q}_0^{2}.\fish
\]
\end{prop}

The lower bound of the remaining eigenvalues is obtained directly from the quadratic form associated to $L$.

\begin{lem}
\label{lower-b}
Suppose $q\in L_{0}^2(\T,\C)$, then
\[
  -(1+\n{q}_0)\n{q}_0\le \Re \lm_{0}^{+}.\fish
\]
\end{lem}
\begin{proof}
Suppose $f\in H^2([0,2])$ is an $L^{2}$-normalized eigenfunction of $\lm_{0}^{+}$, then
\[
  \lm_{0}^{+} = \lin{Lf,f} = \lin{f_x,f_x} + \lin{qf,f},
\]
with the $L^{2}$-inner product $\lin{f,g} \defl \frac{1}{2}\int_{0}^{2} f\ob{g}\,\dx$. With Cauchy-Schwarz, the $L^\infty$-estimate $\n{f}_{L^{\infty}} \le \abs{[f]} + \n{f_{x}}_{0}$, and $\abs{[f]} \le\n{f}_{0} = 1$,
\begin{align*}
  \abs{\lin{f,qf}} \le \n{f}_{L^{\infty}}\n{f}_0\n{q}_0 &\le (\abs{[f]} + \n{f_x}_0)\n{q}_0\\
  &\le
  \n{q}_0 + \n{f_x}_0^2 + \frac{1}{4}\n{q}_0^2.
\end{align*}
Consequently,
\begin{align*}
  \Re \lm_{0}^{+} = \lin{f_{x},f_{x}} + \Re \lin{qf,f}
  \ge - \n{q}_0 - \frac{1}{4}\n{q}_0^2
  \ge -(1+\n{q}_0)\n{q}_0.\qed
\end{align*}
\end{proof}

To prove the remaining part of Proposition~\ref{ev-as}, we cover the right complex half-plane with the closed strips
\[
  U_n \defl \setdef{\lm\in\C}{\abs{\Re \lm-n^{2}\pi^{2}}\le 12n},\qquad n\ge 1.
\]
For $q = 0$ the eigenspace of the double eigenvalue $\lm_{n}^{+} = \lm_{n}^{-} = n^{2}\pi^{2}$, $n\ge 1$, is spanned by $e_{n} \defl \e^{\ii n\pi x}$ and $e_{-n} \defl \e^{-\ii n\pi x}$. When $n$ is sufficiently large, then for $\lm\in U_{n}$ the dominant modes of a solution of $Lf = \lm f$ are thus expected to be $e_{\pm n}$. Therefore, it makes sense to separate these modes from the others by a Lyapunov-Schmidt reduction. To this end, denote the space of \emph{2-periodic} complex-valued $\Hs^{w}$-functions by $\Hs_{\star,\C}^{w}$ and consider the splitting
\begin{align*}
  \Hs_{\star,\C}^{w}
   &= \Pc_n\oplus \Qc_n\\
   &=
  \operatorname{sp}\setdef{e_{k}}{\abs{k}=n}
   \oplus
  \overbar{\operatorname{sp}}\setdef{e_{k}}{\abs{k}\neq n}.
\end{align*}
The projections onto $\Pc_n$ and $\Qc_n$ are denoted by $P_n$ and $Q_n$,
respectively.

We write the eigenvalue equation $Lf = \lm f$ in the form
\[
  A_\lm f \defl f'' + \lm f = V f,
\]
where $V$ denotes the operator of multiplication with $q$.
Since $A_\lm$ is a Fourier-multiplier, by writing
\[
  f = u + v = P_nf + Q_nf,
\]
we can decompose the equation $A_\lm f = V f$ into the two equations
\begin{align*}
  &A_\lm u = P_n V(u+v),\\
  &A_\lm v = Q_n V(u+v),
\end{align*}
called the $P$- and the $Q$-equation.

We first solve the $Q$-equation on each strip $U_n$ by writing $V v$ as a function of $u$. With foresight to estimating the gap lengths, we consider operator norms induced by \emph{shifted weighted norms} \cite{Poschel:2011iua}. For $u$ in $\Hs^w_{\star,\C}$ the $i$-shifted $\Hs^w$-norm is given by
\[
\n{u}_{w;i}^2
 \defl \n{u\e_i}_{w}^2
 = \sum_{k\in\Z} w_{k+i}^{2} \abs{u_k}^2,
 \qquad i\in\Z.
\]

%: Tn-est
\begin{lem}
\label{Tn-est}
If $q\in \Hs_{0,\C}^{w}$, then for any $n\ge 1$ and $\lm\in U_{n}$,
\[
  T_{n} = VA_{\lm}^{-1}Q_{n}
\]
is a bounded linear operator on $\Hs^{w}_{\star,\C}$ with norm
\[
  \n{T_{n}}_{w;i} \le \frac{2}{n}\n{q}_{w},\qquad i\in\Z.\fish
\]
\end{lem}

\begin{proof}
Write $T_n f = V g$ with $g = A_\lm^{-1}Q_nf$. Since $A_\lm e_{m} = (\lm-m^{2}\pi^{2})e_m$, one checks that for $\abs{m}\neq n$
\[
  \min_{\lm\in U_n} \abs{\lm-m^{2}\pi^{2}} \ge \abs{n^{2}-m^{2}} \ge 1.
\]
Hence, the restriction of $A_\lm$ to $\Qc_n$ is boundedly invertible, and the function
\[
  g = A_\lm^{-1}Q_nf = \sum_{\abs{m}\neq n}\frac{f_m}{\lm-m^{2}\pi^{2}}e_m.
\]
is well defined. By Hölder's inequality we obtain for the weighted $\ell^1$-norm
\[
  \n{g\e_{i}}_{\ell^1_w}
  =
  \sum_{\abs{m}\neq n} \frac{w_{m+i}\abs{f_m}}{\abs{\lm-m^{2}\pi^{2}}}
  \le \biggl(\sum_{\abs{m}\neq n} \frac{1}{\abs{n^{2}-m^{2}}^2}\biggr)^{1/2}
      \n{f}_{w;i},
\]
uniformly for $\lm\in U_n$. Moreover,
\[
  \sum_{\abs{m}\neq n} \frac{1}{\abs{n^{2}-m^{2}}^2}
   \le \frac{2}{n^{2}}\sum_{m\ge 1}\frac{1}{m^{2}}
   \le \frac{4}{n^{2}}.
\]
Finally, by Young's inequality for the convolution of sequences
\[
  \n{T_nf}_{w;i} = \n{q(ge_{i})}_{w}
  \le \n{q}_w\n{g\e_i}_{\ell^1_w}
  \le \frac{2}{n}\n{q}_w\n{f}_{w;i}.\qed
\]
\end{proof}

Consequently, $T_n$ is a $\frac 12$-contraction on $\Hs_{\star,\C}^{w}$ for all $n\ge 4\n{q}_w$. If we multiply the $Q$-equation from the left by $V A_\lm^{-1}$, then
\begin{align*}
V v
= V A_\lm^{-1} Q_nV(u+v)
= T_nV(u+v),
\end{align*}
which may be written as
\[
(\Id - T_n)V v = T_nV u.
\]
Hence, for $n\ge 4\n{q}_w$ one finds a unique solution
\[
  V v = \hat{T}_nT_n V u,\qquad \hat{T}_{n} \defl (I-T_{n})^{-1}
\]
of the $Q$-equation. Substituting this solution into the $P$-equation yields
\[
  A_\lm u
   =
  P_n (\Id + \hat{T}_nT_n)V u
   =
  P_n\hat{T}_nV u.
\]
Writing the latter as
\[
  S_n u = 0,\qquad S_n \defl A_\lm - P_n\hat{T}_nV,
\]
we immediately conclude that there exists a one-to-one relationship between a
nontrivial solution of $S_n u = 0$ and a nontrivial 2-periodic solution
of $Lf = \lm f$. Hence, a complex number $\lm\in U_n$ is a
periodic eigenvalue of $L$ if and only if the determinant of
$S_n$ vanishes.

Recall that $P_n$ is the orthogonal projection onto the two-dimensional space
$\Pc_n$. The matrix representation of an operator $B$ on $\Pc_n$ is given by
\[
  \left(\lin{Be_{\pm n},e_{\pm n}}\right)_{\pm,\pm}.
\]
Therefore, we find for $S_n$ the representation
\begin{align*}
  A_\lm
   =
  \mat[\Bigg]{
   \lm - \sg_{n}\\
  &\lm - \sg_{n}
  },
  \qquad\qquad
  P_n\hat{T}_n\Phi =
  \mat[\Bigg]{
  a_n & c_n\\
  c_{-n} & a_{-n}
  },
\end{align*}
with $\sg_{n}=n^{2}\pi^{2}$, and the coefficients of the latter matrix given by
\begin{align*}
  a_n  \defl \lin{\hat{T}_nV e_n,e_n},\qquad
   c_n \defl \lin{\hat{T}_nV e_{-n},e_n}.
\end{align*}
Moreover, by inspecting the expansions of $a_{n}$ and $a_{-n}$ using the representations of $T_{n}$ and $V$ in Fourier space, one concludes that
\[
  a_{n} = \lin{\hat{T}_{n}Ve_{n},e_{n}} = \lin{\hat{T}_{n}Ve_{-n},e_{-n}} = a_{-n}.
\]
%Moreover, by inspecting the series expansion of $\hat{T}_{n}$ one checks that the adjoint $(\hat{T}_{n}V)^{*}$ equals $(\hat{T}_{n}V)^{-}$, the complex conjugate of $\hat{T}_{n}V$ . Therefore,
%\begin{align*}
%  a_{n}
%  &= \lin{\hat{T}_{n}Ve_{n},e_{n}}\\
%  &= \lin{e_{n},(\hat{T}_{n}V)^{-}e_{n}}\\
%  &= \lin{e_{n},(\hat{T}_{n}Ve_{-n})^{-}}\\
%  &= \lin{\hat{T}_{n}Ve_{-n},e_{-n}} = a_{-n}.
%\end{align*}
Hence the diagonal of $S_{n}$ is homogenous,
\[
  S_{n} = \mat[\Bigg]{\lm - \sg_{n} - a_{n} & -c_{n}\\ -c_{-n} & \lm - \sg_{n} -a_{n}}.
\]

We introduce the following notion for the sup-norm of a complex-valued function
\[
  \abs{f}_U \defl \sup_{\lm\in U} \abs{f(\lm)}.
\]

\begin{lem}
\label{coeff-est}
If $n\ge 4\n{q}_{w}$, then $\abs{a_{n}}_{U_{n}} \le 2\n{T_{n}}_{w;n}\n{q}_{w}$ and
\[
  w_{2n}\abs{c_{\pm n}-q_{\pm n}}_{U_{n}} \le 2\n{T_{n}}_{w;\pm n}\n{q}_{w}.\fish
\]
\end{lem}

\begin{proof}
With $\hat{T}_{n} = I + T_{n}\hat{T}_{n}$ we obtain $c_{n} = q_{n} + \lin{T_{n}\hat{T}_{n}Ve_{-n},e_{n}}$. Furthermore, since $\lin{f,\e_n} = \lin{f\e_n,\e_{2n}}$ for any function $f$, we conclude
\[
  w_{2n}\abs{\lin{f,\e_{n}}}
  \le \norm{f e_{n}}_{w} = \norm{f}_{w;n}.
\]
The claim follows with
\[
  \n{\hat{T}_{n}T_{n}V\e_{-n}}_{w;n} \le 2\n{T_{n}}_{w;n} \n{V\e_{-n}}_{w;n},
\]
and $\n{V\e_{-n}}_{w;n} = \n{V\e_{0}}_{w} = \n{q}_{w}$. The proof for $a_{n}$ and $c_{-n}$ is the same.\qed
\end{proof}

The preceding lemma implies that the determinant of $S_n$
\[
\det S_n = (\lm - \sg_{n} - a_n)^2 - c_{n}c_{-n}
\]
is an analytic function in $\lm$, which is close to $(\lm-\sg_{n})^2$ for $n$ sufficiently large. This is what we need to localize the eigenvalues.

\begin{lem}
\label{Sn-roots}
If $n\ge 4\n{q}_{0}$, then the determinant of $S_{n}$ has exactly two complex roots $\xi_{-},\xi_{+}$ in $U_{n}$, which are contained in
\[
  D_{n} = \setdef{\lm}{\abs{\lm-\sg_{n}}\le 4\n{q}_{0}}.\fish
\]
\end{lem}
\begin{proof}
Let $h = \lm - \sg_{n} - a_n$. The preceding lemma together with Lemma~\ref{Tn-est} gives
\[
\abs{a_n}_{U_n} \le \n{q}_{0} < \inf_{\lm\in U_n\setminus D_n} \abs{\lm-\sg_{n}}.
\]
Thus it follows from Rouché's Theorem that $h$ has a single root in $D_n$, just as $(\lm-\sg_{n})$. In a similar fashion, we infer from
\[
  \abs{c_{n}c_{-n}}_{U_{n}} \le 4\n{q}_{0}^{2}  < \inf_{\lm\in U_n\setminus D_n} \abs{h}^{2},
\]
that $h^2$ and $\det S_n$ have the same number of roots in $D_n$, namely two, while $\det S_n$ clearly has no root in $U_n\setminus D_n$.\qed
\end{proof}

\begin{proof}[Proof of Proposition~\ref{ev-as}.]
For each $n \ge 4\n{q}_0$ Lemma~\ref{Sn-roots} applies giving us two roots $\xi_+$ and $\xi_-$ of $\det S_n$ which are contained in $D_n\subset U_n$. Since the strips $U_n$ cover the right complex halfplane, and $\lm_{n}^{\pm}\sim n^{2}\pi^{2}$ as $n\to \infty$, it follows by a standard counting argument that these roots have to be the periodic eigenvalues $\lm_n^\pm$. Thus
\[
  \abs{\lm_{n}^{\pm} - \sg_{n}} \le 4\n{q}_{0},\qquad
  n\ge 4\n{q}_{0}.
\]
Moreover, these are the only eigenvalues contained in $\bigcup_{n \ge 4\n{q}_0} U_n$. Hence if we choose $N\ge 1$ such that $N\ge 4\n{q}_0 > N-1$ then for any $1\le n < N$,
\[
  \Re \lm_{n}^{\pm} \le N^2\pi^2-12N \le 16(N-1)^{2} \le 256\n{q}_0^2.\qed
\]
\end{proof}

We now turn our attention to estimating the gap lengths $\gm_{n} = \lm_{n}^{+}-\lm_{n}^{-}$, which by the preceding considerations satisfy for $n\ge 4\n{q}_{0}$
\[
  \gm_{n}^{2} = (\xi_{+}-\xi_{-})^{2}
\]
with $\xi_{\pm}$ being the complex roots of $\det S_{n}$ on $U_{n}$.

\begin{lem}
If $n\ge 4\n{q}_{0}$, then
\[
  \abs{\xi_+-\xi_-}^{2} \le 9\abs{c_{n}c_{-n}}_{U_{n}}.\fish
\]
\end{lem}

\begin{proof}
We write $\det S_n=g_+g_-$ with
\begin{align*}
g_{\pm} = \lambda - \sg_{n} - a_n \mp \ph_n,\qquad \ph_{n} = \sqrt{c_{n}c_{-n}},
\end{align*}
where the branch of the root is immaterial. Each root $\xi$ of $\det S_n$ is either a root of $g_+$ or $g_-$, respectively, and thus satisfies $\xi = \sg_{n} + a_n(\xi) \pm \ph_n(\xi)$. To estimate the distance of the two roots $\xi_\pm$, we note that $a_n$ is analytic on $U_n$, so Cauchy's estimate gives
\[
  \abs{\partial_\lambda a_n}_{D_n}
   \le
  \frac{\abs{a_n}_{U_n}}{\dist(D_n, \partial U_n)}
   \le
  \frac{\n{q}_{0}}{12 n-4\n{q}_{0}} \le \frac{1}{8}.
\]
The claim now follows with
\begin{align*}
  \abs{\xi_+-\xi_-}
   &\le
  \abs{a_n(\xi_+)-a_n(\xi_-)} + \abs{\ph_n(\xi_+)\pm\ph_n(\xi_-)}\\
   &\le
  \frac{1}{8}\abs{\xi_+-\xi_-} + 2\abs{\ph_n}_{U_n}.\qed
\end{align*}

\end{proof}

The coefficients $c_{\pm n}$ have the same decay as the Fourier coefficients of $q_{\pm n}$, hence the regularity of the potential is reflected in the decay of its gap lengths.

\begin{prop}
\label{gap-est}
If $q\in\Hs_{0,\C}^w$, then for all $N\ge 4\n{q}_w$,
\[
  \sum_{n\ge N} w_{2n}^2\abs{\gm_{n}(q)}^{2} \le 9\n{\Rc_{N}q}_{w}^{2} + \frac{576}{N}\n{q}_{w}^{4},
\]
and further $w_{2n}\abs{\gm_n} \le 6\n{q}_w$ for all $n \ge N$.\fish
\end{prop}

\begin{proof}
The preceding lemma gives for $n \ge 4\n{q}_w$
\[
  \abs{\gm_{n}}^{2} = \abs{\xi_{+}-\xi_{-}}^{2}
   \le 9\abs{c_{n}c_{-n}}
   \le \frac{9}{2}\abs{c_{n}}^{2}_{U_{n}} + \frac{9}{2}\abs{c_{-n}}^{2}_{U_{n}},
\]
and by lemmas \ref{Tn-est} and \ref{coeff-est}
\[
  w_{2n}\abs{c_{n}}
   \le w_{2n}\abs{q_{n}} + 2\n{T_{n}}_{w;-n}\n{q}_{w}
   \le w_{2n}\abs{q_{n}} + \frac{4}{n}\n{q}_{w}^{2}.
\]
Consequently,
\[
  \frac{1}{9}w_{2n}^2\abs{\gm_{n}}^{2}
   \le w_{2n}^2\abs{q_{n}}^2+w_{2n}^2\abs{q_{-n}}^2
   + \frac{32}{n^{2}}\n{q}_{w}^{4},
\]
and summing up gives the claim.\qed
\end{proof}

\begin{cor}
If $q\in\Hs_{0,\C}^w$ is real-valued, then
\[
  \sum_{n\ge 1} w_{2n}^2\abs{\gm_{n}(q)}^{2}
   \le
   2^{18}\p*{ \p*{w(8\n{q}_w)}^2\n{q}_w^{2} + \n{q}_w }\n{q}_w^{2}
   .\fish
\]
\end{cor}
\begin{proof}
Choose $N\ge1$ with $N \ge 4\n{q}_{w} > N-1$, and consider the splitting,
\[
  \sum_{n\ge 1} w_{2n}^{2}\abs{\gm_{n}}^{2} =
  \sum_{1\le n < N} w_{2n}^{2}\abs{\gm_{n}}^{2}
  +
  \sum_{n\ge N} w_{2n}^{2}\abs{\gm_{n}}^{2}.
\]
For the first term we use $\abs{\lm_{N-1}^{+} - \lm_{0}^{+}} \le 261\n{q}_{w}^{4}$, to obtain
\begin{align*}
  \sum_{1\le n < N} w_{2n}^{2}\abs{\gm_{n}}^{2}
  &\le
  w_{2N-2}^2 \paren[\Bigg]{\sum_{1\le n < N} \abs{\gm_{n}}}^{2}\\
  &\le
  \p*{w(8\n{q}_w)}^2\n{q}_w^{2} \abs{\lm_{N-1}^{+}-\lm_{0}^{+}}^{2}\\
  &\le
  (261)^2 \p*{w(8\n{q}_w)}^2\n{q}_w^{2} \n{q}_{w}^{4}.
\end{align*}
The second term can be bounded by Proposition \ref{gap-est},
\[
  \sum_{n\ge N} w_{2n}^2\abs{\gm_{n}}^{2} \le (9+144\n{q}_{w})\n{q}_{w}^{2}.\qed
\]
\end{proof}

\section{Appendix - Asymptotic expansion of the discriminant}
\label{a:dl-exp}

In this appendix we obtain an asymptotic expansion of the discriminant $\Dl(\lm) = \Dl(\lm,q)$ of $-\partial_{x}^{2} + q$ for real-valued potentials in $\Hs_{0}^{N}$ as $\lm\to \infty$ along the real axis. Such an expansion is needed in Section~\ref{s:setup}. It is based on special solutions of $-y'' + qy = \lm y$ studied in \cite{Kappeler:2012fa}.

To state the asymptotics for $\Dl(\lm)$ we first need to recall the definition of the \kdv hierarchy -- cf. \cite[p. 209]{Kappeler:2003up}
\[
  \H_{0} = \frac{1}{2}\int_{\T} q^{2}\,\dx,
  \qquad
  \H_{1} = \frac{1}{2}\int_{\T} ((\partial_{x}q)^{2} + 2q^{3})\,\dx,
  \qquad \dotsc.
\]
For $m\ge 0$ arbitrary, one has
\begin{equation}
  \label{Hm-Sm}
  \H_{m} = \frac{(-1)^{m+1}}{2} S_{2m+3},
  \qquad
  S_{n} \defl \int_{0}^{1} s_{n}(x)\,\dx,
\end{equation}
with $s_{n}$, $n\ge 1$, given recursively as in \cite[Section 2]{Kappeler:2012fa} by $s_{1} = q$, $s_{2} = -\partial_{x}q$ and (if $N\ge 2$),
\begin{equation}
  \label{sm}
  s_{n+1} = -\partial_{x}s_{n} - \sum_{k=1}^{n-1} s_{n-k}s_{k},\qquad 1\le n\le N.
\end{equation}
According to \cite{Kappeler:2012fa} the $s_{n}(x)$ are 1-periodic in $x$, $S_{2m} = 0$ for all $m\ge 1$, and
\begin{equation}
  \label{Sm}
  S_{1} = \int_{\T} q\,\dx,\quad
  S_{3} = -\int_{\T} q^{2}\,\dx,\quad
  S_{5} = \int_{\T} ((\partial_{x}q)^{2} + 2q^{3})\,\dx,\quad
  \dotsc.
\end{equation}

\begin{lem}
\label{dl-exp}
Suppose $q\in \Hs_0^{N}$ with $N\ge 0$, then
\[
  \Dl(\nu^{2}) = 2\cos\Th_{N}(\nu) + O(\nu^{-N-1}),
  \qquad
  \Th_{N}(\nu) \defl \nu - \sum_{3\le 2m+3 \le N} \frac{\H_{m}}{4^{m+1}\nu^{2m+3}},
\]
uniformly in $\nu$ on $\setdef{\nu\in \C}{\abs{\nu}\ge \rho,\quad \abs{\Im \nu}\le \tau}$ with $\rho$, $\tau > 0$ and uniformly in $q$ on bounded subsets of $\Hs_{0}^{N}$.\fish
\end{lem}

\begin{proof}
In~\cite{Kappeler:2012fa}, special solutions of $-y'' + qy = \lm y$ with $\lm = \nu^{2}$ of the following form have been studied:
\[
  z_{N}(x,\nu) = w_{N}(x,\nu) + \frac{r_{N}(x,\nu)}{(2\ii \nu)^{N+1}},
\]
with
\begin{equation}
  \label{wN}
  w_{N}(x,\nu) \defl \exp\left(\int_{0}^{x} \al_{N}(s,\nu)\,\ds\right),
  \quad
  \al_{N}(x,\nu) \defl \ii\nu + \sum_{1\le m\le N} \frac{s_{m}(x)}{(2\ii\nu)^{m}},
\end{equation}
where the $s_{m}$ are given in \eqref{sm}. The error terms $r_{N}(x,\nu)$ satisfy the estimates
\begin{equation}
  \label{rN-bound}
  r_{N}(x,\nu),\; \partial_{x}r_{N}(x,\nu) = O(1),
\end{equation}
uniformly in $0\le x \le 1$, $\nu\in\C$ with $\abs{\Im \nu} \le \tau$, and uniformly on bounded subsets of $\Hs_{0}^{N}$. Given $\tau > 0$ there exists $\rho > 0$ such that for $\abs{\nu} \ge \rho$,
\[
  \xi_{N}(\nu) = \al_{N}(0,\nu)-\al_{N}(0,-\nu)
\]
is uniformly bounded away from zero. This implies that the solutions $z_{N}(x,\nu)$ and $z_{N}(x,-\nu)$ are linearly independent. The fundamental solutions $y_{1}(x,\nu^{2})$, $y_{2}(x,\nu^{2})$ of $-y'' + qy = \nu^{2}y$ can then be expressed as
\begin{align*}
  y_{1}(x,\nu^{2}) &= \frac{1}{\xi_{N}(\nu)}(\al_{N}(0,\nu)z_{N}(x,-\nu) - \al_{N}(0,-\nu)z_{N}(x,\nu)),\\
  y_{2}(x,\nu^{2}) &= \frac{1}{\xi_{N}(\nu)}(z_{N}(x,\nu) - z_{N}(x,-\nu)),
\end{align*}
and it follows from \eqref{rN-bound} that
\[
  \xi_{N}(\nu)y_{1}(x,\nu^{2})
   = \al_{N}(0,\nu)w_{N}(x,-\nu) - \al_{N}(0,-\nu)w_{N}(x,\nu) + O(\nu^{-N-1}).
\]
Note that
\[
  z_{N}'(x,\nu) = \al_{N}(x,\nu)w_{N}(x,\nu) + \frac{r_{N}'(x,\nu)}{(2\ii \nu)^{N+1}},
\]
and hence
\[
  \xi_{N}(\nu)y_{2}'(x,\nu^{2})
   = \al_{N}(x,\nu)w_{N}(x,\nu) - \al_{N}(x,-\nu)w_{N}(x,-\nu) + O(\nu^{-N-1}).
\]
Since $\al_{N}(x,\nu)$ is 1-periodic in $x$, and $\xi_{N}(\nu)$ is uniformly bounded away from zero for $\abs{\nu} > \rho$, the discriminant $\Dl(\nu^{2}) = y_{1}(1,\nu^{2}) + y_{2}'(1,\nu^{2})$ can be computed as
\[
  \Dl(\nu^{2}) =
  w(1,\nu) + w(1,-\nu)
  + O(\nu^{-N-1}).
\]
By definition \eqref{wN} of $w_{N}$ we have $w_{N}(1,\nu) + w_{N}(1,-\nu) = 2\cos(\Th_{N}(\nu))$, where
\[
  \Th_{N}(\nu) = \nu - \ii \sum_{1\le m\le N} \frac{S_{m}}{(2\ii \nu)^{m}}.
\]
Since all $S_{2m}$ vanish, this formula for $\Th_{n}$ simplifies,
\[
  \Th_{N}(\nu) = \nu - \sum_{1\le 2k+1\le N} \frac{(-1)^{k}\frac{1}{2}S_{2k+1}}{4^{k}\nu^{2k+1}},
\]
and after changing the index of summation $k=n+1$ one concludes from \eqref{Hm-Sm} and the assumption $S_{1} = 0$ that
\[
  \Th_{N}(\nu) = \nu - \sum_{1\le 2n+3\le N} \frac{H_{n}}{4^{n+1}\nu^{2n+3}}
\]
as claimed. Going through the arguments of the proof one verifies that the claimed uniformity statements hold.\qed
\end{proof}

% ==================================================================================================

\section{\boldmath Appendix - Analyticity of $F$}
\label{a:ab-int}

In this appendix we prove analyticity properties of the functionals $F_{n}$, $n\ge 1$, introduced in Section~\ref{s:setup}. We use the notation established in Section~\ref{s:setup}. In particular, the open neighborhoods $\Wp$ and $\Wp_{q}$.

\begin{lem}
\label{F-ana}
Let $q \in \Wp$ and $n \ge 1$. Then for any $\nu \in \partial U_n$,
\[
  F_n(\nu) =  \frac{1}{2}\left(\int_{\lm_n^-}^\nu \om + \int_{\lm_n^+}^\nu \om\right),
  \qquad \om \defl \frac{\dDl}{\sqrt[c]{\Dl^{2}-4}}\,\dlm
\]
is analytic on $\Wp_q$.\fish
\end{lem}

\begin{rem*}
By the same line of arguments one obtains the analyticity of $F_{n}$ when $\Ws$ is extended to the open neighborhood of $\Hs^{-1}_0$ within $\Hs^{-1}_{0,\C}$ constructed in \cite{Kappeler:2005fb}, and  $\Ws_q$ is chosen as an open neighborhood of $q$ within $\Hs^{-1}_{0,\C}$ such that a common set of isolating neighborhoods exists.
\end{rem*}

\begin{proof}
We want to apply \cite[Theorem A.6]{Grebert:2014iq}. First note that for any $\mu, \nu \in \partial U_n$, $\int_\nu^\mu \om$ is analytic on $\Wp_q$ by Lemma~\ref{w-closed}. Hence at each step of the proof we might change $\nu$ at our convenience.

In a first step we prove that $F_n$ is analytic on $\Wp_q \setminus Z_n$, where $Z_{n} = \setdef{p\in \Wp}{\gm_{n}^{2}(p) = 0}$. Note that $\gm_{n}^{2}$ is analytic on $\Wp$ thus $Z_{n}$ is an analytic subvariety of $\Wp$. Fix $p_0 \in \Wp_q \setminus Z_n$. Recall that $\lm_n^-$, $\lm_n^+$ are lexicographically ordered and might therefore not even be continuous near $p_0$. However, since the eigenvalues $\lm_n^+(p_0)$, $\lm_n^-(p_0)$ are both simple, there exist a neighborhood $\Vs\subset \Wp_q \setminus Z_n$ of $p_0$ and two analytic functions $\rho_1, \rho_2\colon \Vs \to \C$ satisfying the set inequality $\setd{ \rho_1, \rho_2 } = \setd{ \lm_n^-, \lm_n^+ }$ on $\Vs$ and in addition
\begin{align*}
  & \abs{\rho_{1} - \lm_n^-(p_0)} <
    \frac{1}{4}\min\left[
      \dist\p*{\lm_{n}^{-}(p_{0}}, \partial U_{n}), \abs{\gm_{n}(p_{0})}
    \right], \\
  & \abs{\rho_{2} - \lm_n^+(p_0)} <
    \frac{1}{4}\min\left[
      \dist\p*{\lm_{n}^{+}(p_{0}}, \partial U_{n}), \abs{\gm_{n}(p_{0})}
    \right].
\end{align*}
In particular, $\rho_1(p_0) = \lm_n^-(p_0)$ and $\rho_2(p_0) = \lm_n^+(p_0)$. Let $\nu_1$, $\nu_2 \in \partial U_n$ be elements of the form $\nu_1 = \lm_n^-(p_0)  - \sigma_1 \gm_n(p_0)$, $\nu_2 = \lm_n^+(p_0)  + \sigma_2 \gm_n(p_0)$ with $\sigma_1$, $\sigma_2 >0$. Then for each $p \in \Vs$, the line segments $[\rho_1, \nu_1]$ and $[\rho_2, \nu_2]$ are admissible paths and
\[
  2 F_n(\nu) = \int_{\rho_1}^{\nu_1}\om + \int_{\nu_1}^{\nu}\om
   + \int_{\rho_2}^{\nu_2}\om + \int_{\nu_2}^{\nu}\om.
\]
As already noted above, $\int_{\nu_1}^{\nu}\om$ and $\int_{\nu_2}^{\nu}\om$
are both analytic on $\Wp_q$. We now prove that for $i=1,2$, $\int_{\rho_i}^{\nu_i}\om $ is analytic on $\Vs$. Since the cases $i=1$ and $i=2$ are treated in the same way we concentrate on $i=1$. With the parametrization $\lm_t =  \rho_1 + t(\nu_1 - \rho_1)$ one gets
\[
  \int_{\rho_1}^{\nu_1}\om
   = \frac{\nu_1 - \rho_1}{2\ii}
     \int_0^1 \, \chi_n(\lm_t) \frac{\lm_n^\ld - \lm_t}{\vs_{n}(\lm_{t})} \, \dt,
\]
where
\begin{equation}
  \label{chi-n}
    \chi_n(\lm)
   =\frac{1}{\sqrt[+]{\lm - \lm_{0}^{+}}} \, \prod_{1\le m\neq n} \frac{\lm_m^\ld - \lm}{\vs_{m}(\lm)}
\end{equation}
is analytic on $U_{n}\times \Wp_{q}$ -- see e.g. \cite[Appendix A]{Kappeler:2005fb}.
Taking into account that
\[
  (\tau_{n}-\lm_{t})^{2} - \gm_{n}^{2}/4
   = (\rho_1 - \lm_t)(\rho_2 - \lm_t)
   = t(\rho_1 - \nu_1) (\rho_2 - \lm_{t}),
\]
one obtains
\[
  \vs_{n}(\lm_{t})
   = (\tau_{n}-\lm_{t})\sqrt[+]{t}\sqrt[+]{\frac{(\rho_1 - \nu_1) (\rho_2 - \lm_{t})}{(\tau_{n}-\lm_{t})^{2}}}.
\]
Moreover, at $p_{0}$ we have for any $0\le t \le 1$,
\[
  \frac{(\rho_1 - \nu_1) (\rho_2 - \lm_{t})}{(\tau_{n}-\lm_{t})^{2}}
   = \frac{\sg_{1}(1+t\sg_{1})\gm_{n}^{2}}{(1/2+t\sg_{1})^{2}\gm_{n}^{2}}
   = \frac{\sg_{1}(1+t\sg_{1})}{(1/2+t\sg_{1})^{2}}
   \ge \frac{\sg_{1}}{1/2+\sg_{1}} > 0.
\]
Thus, after possibly shrinking $\Vs$, the mappings $(t,p)\mapsto \sqrt[+]{\frac{(\rho_1 - \nu_1) (\rho_2 - \lm_{t})}{(\tau_{n}-\lm_{t})^{2}}}$ and $(t,p) \mapsto \tau_n - \lambda_t$ are  continuous on $[0,1]\times \Vs$, analytic on $\Vs$ for every fixed $0\le t\le 1$, and uniformly bounded away from zero by some $c_{1} > 0$. Since $\chi_{n}$ is analytic on $U_{n}\times\Wp_{q}$, and $\rho_1, \rho_2$, and $\lm_n^\ld$ are analytic on $\Vs\subset\Wp_{q}$, it then follows that
\[
  \int_{\rho_1}^{\nu_1}\om
   = \frac{(\nu_1 - \rho_1)}{2\ii} \int_0^1 \, \frac{\chi_n(\lm_t) }{\tau_n - \lambda_t}
   \frac{\lm_n^\ld - \lm_t}{\sqrt[+]{\frac{(\rho_1 - \nu_1) (\rho_2 - \lm_{t})}{(\tau_{n}-\lm_{t})^{2}}}} \, \frac{\dt}{\sqrt[+]{t}},
\]
is analytic on $\Vs$ as well.

In a second step we show that the restriction of $F_n$ to $Z_n\cap \Wp_{q}$ is weakly analytic. Note that on $Z_n\cap \Wp_{q}$, $F_n$ coincides with the function $\frac{1}{2\ii}\int_{\tau_n}^\nu \chi_n(\lm) \, \dlm$ with $\chi_n$ given by \eqref{chi-n}, where the path of integration is chosen in $U_n$ but otherwise arbitrary. Thus $F_n\vert_{Z_n\cap \Wp_{q}}$ is weakly analytic.

In a third and final step we prove that $F_n$ is continuous on $\Wp_q$. By the considerations above, $F_n$ is continuous in each point of $\Wp_q \setminus Z_n$ and the restriction of $F_n$ to $Z_n\cap \Wp_{q}$ is continuous. Hence it remains to show that for any $p_{\infty} \in Z_n\cap \Wp_{q}$ and any sequence $(p_k)_{k \ge 1} \subset \Wp_q \setminus Z_n$ with $p_k \to p_{\infty}$ in $\Wp_q$ it follows that $F_n(p_k) \to F_n(p_{\infty})$ as $k \to \infty$. By \cite[Proposition~B.13]{Kappeler:2003up} (see also \cite[Proposition~2.9]{Kappeler:2005fb}) $\lm_n^\ld - \tau_n = O(\gm_{n}^{2})$ locally uniformly around $p_{\infty}$. Since $\gm_n(p_k)\to 0$ as $k \to \infty$, there exists $k_{0}\ge 1$ so that
\begin{equation}
  \label{lmd-tau-gm}
  \abs{\lm_n^\ld(p_{k}) - \tau_n(p_{k})} \le \abs{\gm_{n}(p_{k})}/2,\qquad k\ge k_{0}.
\end{equation}
Without loss of generality we may assume that $k_{0} = 1$.
For each $p_{k}$ consider the parametrization of $[\lm_{n}^{-},\tau_{n}]$, given by $[0,1]\to G_{n}$, $t\mapsto \lm_{t} =  \tau_n - t \gm_n/2$.
Then $\vs_n(\lm_t)^2 = -(1-t^2) \gm_n^2/4$ and thus
\[
  \abs*{\int_{\lm_n^-}^{\tau_n} \frac{\lm_n^\ld - \lm}{\vs_n(\lm)} \, \chi_n(\lm) \, \dlm}
  \le \int_0^1 \frac{\abs{\lm_n^\ld - \tau_n + t\gm_n/2}}{\sqrt[+]{1-t^2}} \, \abs{\chi_n(\lm_t)} \, \dt \le C \abs{\gm_n},
\]
where $C>0$ can be chosen independently of $k$. Since $\gm_n(p_k)\to 0$ as $k\to \infty$, we conclude
\begin{equation}
\label{A.6}
  F_n(\nu,p_k) = \frac{1}{2\ii}\int_{\tau_n}^{\nu } \frac{\lm_n^\ld - \lm}{\vs_n(\lm)} \, \chi_n(\lm) \, \dlm + o(1),\qquad k\to \infty.
\end{equation}
We can choose $\nu_* \in \partial U_n$ so that the straight line segment $[\tau_n, \nu_*]$ is admissible for any $p_k$, $k \ge 1$.
Note that $z_{k} \defl \nu_{*} - \tau_{n}(p_{k})\neq 0$ for any $k\ge 1$ and $z_{k}\to z_{\infty} \defl \nu_{*}-\tau_{n}(p_{\infty})$. With the parametrization $\lm_t  = \tau_n + tz_k$, $0 \le t \le 1$, and since by the definition of the standard root $\vs_{n}(\lm_{t}) = -z_{k}\sqrt[+]{t^{2}-\gm_{n}^{2}/4z_{k}^{2}}$, one then obtains
\[
  \int_{\tau_n}^{\nu_* } \frac{\lm_n^\ld - \lm}{\vs_n(\lm)} \, \chi_n(\lm) \, \dlm
   = z_k\int_0^1 \frac{t + (\tau_{n}-\lm_n^\ld)/z_{k}}{\sqrt[+]{t^{2} - \gm_{n}^{2}/4z_{k}^{2}}} \, \chi_n(\lm_t) \, \dt.
\]
As $\gm_{n}(p_{k})\to 0$, one has $\sqrt[+]{t^{2}-\gm_{n}^{2}/4z_{k}^{2}}\big|_{p_{k}}\to t$ and $t+(\tau_{n}-\lm_{n}^{\ld})\big|_{p_{k}} \to t$  by~\eqref{lmd-tau-gm}. To see that the integral converges to $\int_{\tau_{n}}^{\nu_{*}} \chi_{n}(\lm)\,\dlm\big|_{p_{\infty}}$ let $\ep_{k} = \gm_{n}(p_{k})/2z_{k}$. Then
\[
  \abs*{ \sqrt[+]{t^{2}-\ep_{k}^{2}} }
   \ge \sqrt{t^{2}-\abs{\ep_{k}}^{2}}
   \ge (t-\abs{\ep_{k}})^{1/2}(t+\abs{\ep_{k}})^{1/2},
\]
whereas by \eqref{lmd-tau-gm} $\abs{(\tau_{n}-\lm_{n}^{\ld})/z_{k}} \le \abs{\ep_{k}}$. Hence for any $0\le t \le 1$
\[
  \abs*{\frac{t + (\tau_{n}-\lm_n^\ld)/z_{k}}{\sqrt[+]{t^{2} - \gm_{n}^{2}/4z_{k}^{2}}} \, \chi_n(\lm_t)} \le Cg_{k}(t),\qquad g_{k}(t) =\frac{(t+\abs{\ep_{k}})^{1/2}}{(t-\abs{\ep_{k}})^{1/2}}.
\]
Since $g_{k}$ is in $L^{1}[0,1]$ for any $k\ge 1$ and converges in $L^{1}[0,1]$ to the constant function $1$ as $k\to \infty$, it follows from the generalized dominated convergence theorem that
\begin{align*}
  \lim_{k\to\infty} z_k\int_0^1 \frac{t + (\tau_{n}-\lm_n^\ld)/z_{k}}{\sqrt[+]{t^{2} - \gm_{n}^{2}/4z_{k}^{2}}} \, \chi_n(\lm_t) \, \dt
  &= z_{\infty}\int_{0}^{1} \chi_{n}(\lm_{t})\,\dt\bigg|_{p_{\infty}}\\
  &= \int_{\tau_{n}}^{\nu_{*}} \chi_{n}(\lm)\dlm\bigg|_{p\infty}.
\end{align*}
In view of \eqref{A.6} it then follows that $\lim_{k \to \infty} F_n(p_k) = F_n(p_{\infty})$.
Altogether we have shown that $F_n$ is continuous on $\Wp_q$. The claimed analyticity then follows from \cite[Theorem A.6]{Grebert:2014iq}.\qed
\end{proof}

\section{Appendix - Auxiliary results}

\begin{lem}
\label{ana-vanish}
Let $E_r$ be a real Banach space and denote by $E$ its complexification. Assume that $U \subset E$ is an open connected neighborhood of $U_r = U \cap E_r$ and that $f\colon U \to \C$ is an analytic map. If $f\vert_{U_r} = 0$, then $f \equiv 0$.\fish
\end{lem}

\begin{proof}
Let $u \in U_r$. Then near $u$, $f$ is represented by its Taylor series,
\[
  f(u+h) = \sum_{n \ge 0} \frac{1}{n!} \ddd_u^n f(h, \ldots, h)
\]
which converges absolutely and uniformly (cf. e.g. \cite[Theorem A.4]{Kappeler:2003up}). As $f\vert_{U_r} = 0$, it follows that for any $h \in E_r$ and any $n\ge 1$, $d_u^n f(h, \ldots, h) = 0$. As $f$ is analytic, $d_u^n f$ is symmetric and $\C$-multilinear. Hence it follows that
\[
  d_u^n f(h, \ldots, h) = 0,\qquad h\in E,\quad n\ge 0,
\]
which implies that $f \equiv 0$ in a neighborhood $V_u$ of $u$ with $V_u \subset U$. As $U$ is connected it follows from the identity theorem that $f\equiv 0$.\qed
\end{proof}

\end{appendix}

\end{document}